\documentclass[11pt]{article}
\usepackage[T1]{fontenc}
\usepackage[utf8]{inputenc}
\usepackage{amsmath,amssymb,amsthm}
\usepackage[a4paper,margin=1.1in]{geometry}

\theoremstyle{plain}
\newtheorem{theorem}{Theorem}[section]
\newtheorem{lemma}[theorem]{Lemma}
\newtheorem{proposition}[theorem]{Proposition}
\newtheorem{corollary}[theorem]{Corollary}
\newtheorem{fact}[theorem]{Fact}
\theoremstyle{definition}

\newtheorem{remark}[theorem]{Remark}

\newcommand{\mm}{\mathbb{M}}
\newcommand{\HFM}{\mathrm{HFM}}
\newcommand{\enc}[1]{\langle #1\rangle}   % singleton / enclosure
\newcommand{\U}{\uplus}                   % multiset union / juxtaposition
\newcommand{\sq}{\sqsubseteq}             % containment
\newcommand{\Fm}{\mathsf{F}^{-}}
\newcommand{\WFm}{\mathsf{WF}^{-}}
\newcommand{\Qr}{\mathsf{Q}}
\newcommand{\Rr}{\mathsf{R}}
\newcommand{\TT}{\mathsf{T}}
\newcommand{\WT}{\mathsf{WT}}
\newcommand{\AST}{\mathsf{AST}}
\newcommand{\TC}{\mathsf{TC}}
\newcommand{\QTp}{\mathsf{QT}^{+}}
\newcommand{\intp}{\trianglelefteq}       % "is interpretable in"
\newcommand{\mutint}{\equiv_{\mathrm{I}}} % mutual interpretability
\newcommand{\pr}[2]{\pi(#1,#2)}
\newcommand{\nf}{\mathrm{nf}}
\newcommand{\hull}{\mathrm{hull}}

\title{Weak essentially undecidable theories of\\ hereditarily finite multisets}
\author{Platon Sifnaios\\[2pt]
\small Independent Researcher, Athens, Greece\\[1pt]
\small\texttt{p.sifnaios@gmail.com}}
\date{}

\begin{document}
\maketitle

\begin{abstract}
We introduce two first-order theories of hereditarily finite multisets: a schematic
theory $\WFm$ and a finitely axiomatized theory $\Fm$, in the language with the empty
multiset, singleton formation, multiset union, and a containment relation. We prove
that $\WFm$ is mutually interpretable with Robinson's theory $\Rr$, and $\Fm$ with
Robinson arithmetic $\Qr$; in particular, $\Fm$ is essentially undecidable. Multisets
thereby join numbers, strings, trees, sets, and sequences in the mutual-interpretability
classes of $\Rr$ and $\Qr$. The distinctive obstacle of the multiset case is the
simultaneous failure of the standard devices for recovering ordered pairs: positional
order, local order on immediate constituents, and idempotence-based Kuratowski pairing.
We show that order is recoverable from bare multiplicity: the term
$\pr{x}{y}=\enc{x}\U\enc{x}\U\enc{y}$ is provably injective in $\Fm$, yielding a direct
interpretation of the Kristiansen--Murwanashyaka tree theory $\TT$; conversely, $\Fm$
is interpreted in $\Qr$ by arithmetizing a normal-form calculus for multiset terms
within bounded arithmetic. Each structural axiom of $\Fm$ is shown independent of the
others, with finite or Presburger-definable decidable witnesses, and the containment
axiom is conservative. As an application, we identify Spencer-Brown's forms modulo
commutative juxtaposition with hereditarily finite multisets and locate the boundary
of essential undecidability within the calculus of indications.
\end{abstract}


\section{Introduction}\label{sec:intro}
The monograph of Tarski, Mostowski, and Robinson \cite{TMR53} isolated two weak
arithmetical theories, $\Rr$ and $\Qr$, as basis theories for metamathematical
arguments: both are essentially undecidable, and interpretability of either in a
consistent theory transfers essential undecidability to it. Interpretability is
reflexive and transitive, and thus organizes the recursively enumerable
essentially undecidable theories into degrees. A notable development of the past
two decades is the discovery that the natural weak theories of finitary
structures cluster into just two mutual-interpretability classes, those of $\Rr$
and of $\Qr$. The class of $\Qr$ contains, besides $\Qr$ itself, Grzegorczyk's
concatenation theory $\TC$ \cite{Grz05,GZ08}, shown mutually interpretable with
$\Qr$ through \cite{Gan09,Sve07,Vis09,Dam17}; adjunctive set theory $\AST$
\cite{TMR53,CH70,MM94,Dam17}; the theory $\TT$ of full binary trees of
Kristiansen and Murwanashyaka \cite{KM20}, shown interpretable in $\Qr$ by
Damnjanovic \cite{Dam22}, with $m$-ary generalizations in \cite{Dam23}; and the
sequence theory $\mathsf{Seq}$ \cite{KM24}. The class of $\Rr$ contains, besides
$\Rr$ itself \cite{JS83}, the concatenation theory $\mathsf{WTC}^{-\varepsilon}$
of Higuchi and Horihata \cite{HH14}, the tree theory $\WT$ \cite{KM20,Dam22},
the hybrid theory $\mathsf{WQT}^{*}$ \cite{Dam22}, and weak concatenation
variants \cite{Mur22,Mur24}. Numbers, strings, trees, sets, sequences: each
species of finitary structure treated in this programme is represented --- with
one exception.

Hereditarily finite multisets --- finite multisets whose elements are themselves
hereditarily finite multisets --- are absent from this picture, although they
occupy a distinguished position in the family: they are the free commutative
counterpart of finite ordered trees and the non-idempotent counterpart of
hereditarily finite sets. The present paper fills the gap. We introduce a
schematic theory $\WFm$ and a finitely axiomatized theory $\Fm$ in the language
$L=\{\emptyset,\enc{\cdot},\U,\sq\}$ --- empty multiset, singleton, multiset
union, containment\footnote{We call $\sq$ the \emph{containment} or
\emph{subform} relation. The term ``pervasion'' is avoided: in parts of the
boundary-algebra literature it names an equational deletion principle rather
than an occurrence relation.} --- both true in the intended model $\mm$ of
hereditarily finite multisets, and prove:
\begin{itemize}\setlength\itemsep{2pt}
\item[\textbf{A.}] $\WFm$ is mutually interpretable with $\Rr$
(Theorem~\ref{thm:A});
\item[\textbf{B.}] $\Fm$ is mutually interpretable with $\Qr$; in particular,
$\Fm$ is essentially undecidable (Theorem~\ref{thm:B});
\item[\textbf{C.}] each of the structural axioms $\mathrm{F}_4$--$\mathrm{F}_8$
of $\Fm$ is independent of the remaining axioms, its removal admitting a
consistent decidable complete extension witnessed by a finite or
Presburger-definable model (Theorem~\ref{thm:indep}), while the containment
axiom $\mathrm{F}_9$ is conservative over the rest
(Proposition~\ref{prop:cons}); minimality is understood relative to the
axiomatization, there being no interpretability-minimal recursively enumerable
essentially undecidable theory \cite{MPV24}.
\end{itemize}

Kristiansen and Murwanashyaka have suggested that the unexpected strength of
these weak theories stems from their ability to represent sequences --- that
access to sequences is, intuitively, both necessary and sufficient for
reconstructing substantial mathematics \cite{KM24}. Multisets furnish the
sharpest available test of the necessity half of that intuition, for they are
the order-free finitary structure par excellence. A multiset carries no
positional order, as strings and sequences do; no local order on immediate
constituents, as binary trees do; and, unlike sets, not even the idempotent
collapse that Kuratowski pairing exploits in adjunctive set theory. Every
standard device for recovering ordered pairs --- the currency in which
interpretations of $\Qr$ are paid --- fails simultaneously. Theorem~B shows
that order is nonetheless recoverable from bare multiplicity: the term
$\pr{x}{y}:=\enc{x}\U\enc{x}\U\enc{y}$, the multiset $[x,x,y]$, is provably
injective on the basis of finitely many universal and $\forall\exists$ axioms
(Pairing Lemma~\ref{lem:pairing}). Access to sequences is thus not a primitive
prerequisite for essential undecidability: it is reconstructible from
multiplicity and nesting alone.

The interpretations run as follows. For Theorem~B, the Pairing Lemma yields a
direct interpretation --- identical on equality, without relativization --- of
the tree theory $\TT$ in $\Fm$ (Theorem~\ref{thm:TinF}); composition with the
interpretation of $\Qr$ in $\TT$ from \cite{KM20} gives $\Qr\intp\Fm$. For the
converse, a normal-form calculus for multiset terms is arithmetized inside
Buss's $S^1_2$ \cite{Bus86}: canonical words over the bracket alphabet code
multisets, union is sorted merge of top-level blocks, containment is occurrence
as a balanced segment, and the nine axioms are verified with
$\Delta^b_1$-machinery and $\Sigma^b_1$-induction \cite{Bus86,HP93}; descent to
$\Qr$ is through the mutual interpretability of $S^1_2$ with $\Qr$
\cite{FF13,Dam22}. For Theorem~A, closed-term instances translate schema to
schema: the coding $c(\bot):=\emptyset$,
$c(\langle s,t\rangle):=\pr{c(s)}{c(t)}$ sends distinct trees to provably
distinct values (Lemma~\ref{lem:coding}), and the containment hull of a code
consists exactly of the codes of the subtrees (Lemma~\ref{lem:hullcode}); the
converse direction follows from Visser's characterization of interpretability
in $\Rr$ by local finite satisfiability \cite{Vis14}, witnessed by
weight-truncated models with an absorbing element (Lemma~\ref{lem:MB}) --- the
deliberate omission of cancellation from $\WFm$ is what keeps such finite
models available. The two theorems, with Corollary~\ref{cor:closing}, reproduce
over multisets the exact interpretability configuration of the pair
$(\WT,\TT)$ over trees.

The intended model has an independent pedigree. The forms of Spencer-Brown's
calculus of indications \cite{SB69}, generated from the empty expression by
enclosure and juxtaposition, are --- modulo the associativity, commutativity,
and neutrality that spatial juxtaposition satisfies --- precisely the
hereditarily finite multisets, with enclosure as singleton and juxtaposition as
union (Theorem~\ref{thm:ident1}). Imposing the idempotence law collapses them
onto hereditarily finite sets, where the composite $x\enc{y}$ is, verbatim, set
adjunction (Theorem~\ref{thm:ident2}), and the resulting condensed theory is
mutually interpretable with $\Qr$ through the known chain via adjunctive set
theory (Proposition~\ref{prop:condensed}). Finite-state re-entry, the feedback
mechanism of Chapter~11 of \cite{SB69} and of the waveform tradition
\cite{Var75,KV80}, generates only ultimately periodic behaviour and remains
within the decidable regime of B\"uchi's S1S (Proposition~\ref{prop:reentry};
\cite{Buc62,Elg61}). Within the calculus of indications, the boundary of
essential undecidability is thereby located exactly: it is crossed neither by
multiplicity alone (Presburger \cite{Pre29}, with \cite{FV59}), nor by
multiplication alone (Skolem \cite{Sko30}, Mostowski \cite{Mos52}), nor by
ordered nesting without a subterm relation \cite{Mal61,Mah88,Hod93}, nor by
finite-state re-entry --- but by unbounded containment of form within form
(Section~\ref{sec:calculus}).

All theories considered are theories of their intended models: every axiom of
$\WFm$ and $\Fm$ holds in $\mm$ (Proposition~\ref{prop:soundF}), and the
inclusion of the containment relation in the primitive vocabulary follows the
design of the tree theories $\WT$ and $\TT$, whose language likewise carries
the subterm relation interpreted over a term model \cite{KM20}. The external
inputs of the paper are confined to the results recorded in
Section~\ref{sec:prelim}, the classical decidability theorems cited where used,
the Knaster--Tarski fixpoint theorem \cite{Tar55}, and the standard
definability infrastructure of $S^1_2$ \cite{Bus86,HP93}; each is invoked in
the exact form cited.

Section~\ref{sec:prelim} fixes interpretability preliminaries and records the
external results. Section~\ref{sec:theories} introduces $\WFm$ and $\Fm$,
establishes the normal-form calculus, and shows that $\Fm$ extends $\WFm$
(Proposition~\ref{prop:FextendsWF}). Section~\ref{sec:pairing} proves the
Pairing Lemma and constructs the direct interpretation of $\TT$.
Sections~\ref{sec:thmA} and~\ref{sec:thmB} prove Theorems A and~B,
Section~\ref{sec:minimality} proves Theorem~C, and Section~\ref{sec:calculus}
contains the identification theorems and the boundary map for the calculus of
indications.

\section{Preliminaries}\label{sec:prelim}

\subsection{Theories and interpretability}
Theories are first-order theories with equality, identified with their sets of
non-logical axioms; the language of a theory is the language of its axioms. A theory
is \emph{recursively axiomatized} if its set of axioms is decidable. A consistent
theory $U$ is \emph{essentially undecidable} if no consistent extension of $U$ in the
same language is decidable \cite{TMR53}.

A \emph{(relative) interpretation} of a theory $V$ in a theory $U$ is given by a
translation $\tau$ of the language of $V$ into the language of $U$: a formula
$\delta(x)$ (the \emph{domain}), for which $U$ proves $\exists x\,\delta(x)$; for each
$n$-ary relation symbol of $V$ a formula in $n$ free variables; and for each $n$-ary
function symbol a formula in $n{+}1$ free variables which $U$ proves to be total and
functional on $\delta$. Translations commute with connectives, and quantifiers are
relativized to $\delta$. We write $V\intp U$ (\emph{$V$ is interpretable in $U$}) if
there is a translation $\tau$ such that $U\vdash\varphi^{\tau}$ for every axiom
$\varphi$ of $V$. An interpretation is \emph{direct} if $\delta(x):=(x{=}x)$, equality
is translated as equality, and function symbols are translated by terms; direct
interpretations compose with arbitrary ones. We write $U\mutint V$ for mutual
interpretability. Interpretability is reflexive and transitive.

\begin{fact}[\cite{TMR53}]\label{fact:transfer}
If $V$ is essentially undecidable and $V\intp U$ for a consistent theory $U$, then
$U$ is essentially undecidable.
\end{fact}

\subsection{The arithmetical base theories}
The language of arithmetic here is $\{0,S,+,\cdot,\le\}$; numerals are
$\bar n:=S^{n}0$. The axioms of $\Qr$ and of $\Rr$ are listed in
Figure~\ref{fig:RQ}; in $\Qr$, $x\le y$ abbreviates $\exists z\,(z+x=y)$. Both
theories are essentially undecidable \cite{TMR53}; $\Rr$ is not finitely
axiomatizable, and the schematic presentation of $\Rr$ below follows
\cite{JS83}.

\begin{figure}[t]
\centering
\fbox{\begin{minipage}{0.44\textwidth}
\textbf{Axioms of $\Qr$}
\begin{itemize}\setlength\itemsep{1pt}
\item[$\Qr_1$:] $Sx\neq 0$
\item[$\Qr_2$:] $Sx=Sy\rightarrow x=y$
\item[$\Qr_3$:] $x\neq 0\rightarrow\exists y\,(x=Sy)$
\item[$\Qr_4$:] $x+0=x$
\item[$\Qr_5$:] $x+Sy=S(x+y)$
\item[$\Qr_6$:] $x\cdot 0=0$
\item[$\Qr_7$:] $x\cdot Sy=x\cdot y+x$
\end{itemize}
\end{minipage}}
\hfill
\fbox{\begin{minipage}{0.50\textwidth}
\textbf{Axiom schemes of $\Rr$} (all $n,m\in\mathbb{N}$)
\begin{itemize}\setlength\itemsep{1pt}
\item[$\Rr_1$:] $\bar n+\bar m=\overline{n+m}$
\item[$\Rr_2$:] $\bar n\cdot\bar m=\overline{n\cdot m}$
\item[$\Rr_3$:] $\bar n\neq\bar m$ \quad for $n\neq m$
\item[$\Rr_4$:] $\forall x\,(x\le\bar n\leftrightarrow x=\bar 0\vee\dots\vee x=\bar n)$
\item[$\Rr_5$:] $\forall x\,(x\le\bar n\vee\bar n\le x)$
\end{itemize}
\end{minipage}}
\caption{The theories $\Qr$ and $\Rr$.}\label{fig:RQ}
\end{figure}

\subsection{The reference theories of trees}
The tree theories of Kristiansen and Murwanashyaka \cite{KM20} are formulated in the
language $L_{\TT}=\{\bot,\langle\cdot,\cdot\rangle,\sq\}$ with a constant, a binary
function symbol, and a binary relation symbol. The intended model is the term model:
the universe is the set of variable-free $L_{\TT}$-terms built from $\bot$ and
pairing, and $\sq$ is interpreted as the subterm relation. The finitely axiomatized
theory $\TT$ and the schematic theory $\WT$ are given in Figure~\ref{fig:WTT};
in $(\mathrm{WT}_2)$, $S(t)$ denotes the set of all subterms of $t$. The theory
$\TT$ extends $\WT$ \cite{Dam22}.

\begin{figure}[t]
\centering
\fbox{\begin{minipage}{0.96\textwidth}
\textbf{Axioms of $\TT$}
\begin{itemize}\setlength\itemsep{1pt}
\item[$\TT_1$:] $\forall x,y\;\neg\,(\langle x,y\rangle=\bot)$
\item[$\TT_2$:] $\forall x,y,z,w\;(\langle x,y\rangle=\langle z,w\rangle
      \rightarrow x=z\wedge y=w)$
\item[$\TT_3$:] $\forall x\;(x\sq\bot\leftrightarrow x=\bot)$
\item[$\TT_4$:] $\forall x,y,z\;(x\sq\langle y,z\rangle\leftrightarrow
      x=\langle y,z\rangle\vee x\sq y\vee x\sq z)$
\end{itemize}
\medskip
\textbf{Axiom schemes of $\WT$} (over variable-free $L_{\TT}$-terms $s,t$)
\begin{itemize}\setlength\itemsep{1pt}
\item[$\mathrm{WT}_1$:] $\neg\,(s=t)$ \quad for distinct variable-free terms $s,t$
\item[$\mathrm{WT}_2$:] $\forall x\,\bigl(x\sq t\leftrightarrow
      \bigvee_{s\in S(t)}x=s\bigr)$ \quad for each variable-free term $t$
\end{itemize}
\end{minipage}}
\caption{The tree theories $\TT$ and $\WT$ of \cite{KM20}.}\label{fig:WTT}
\end{figure}

\begin{theorem}[{\cite[Thm.~2]{KM20}}]\label{thm:WTR}
$\WT\mutint\Rr$.
\end{theorem}

\begin{theorem}[{\cite[Thm.~11]{KM20}}]\label{thm:QT}
$\Qr\intp\TT$.
\end{theorem}

\begin{theorem}[Visser \cite{Vis14}]\label{thm:visser}
A recursively axiomatized theory is interpretable in $\Rr$ if and only if it is
locally finitely satisfiable, i.e.\ every finite subset of its non-logical axioms
has a finite model.
\end{theorem}

\begin{theorem}[Damnjanovic \cite{Dam22}; see also \cite{Dam17,FF13}]\label{thm:chains}
$\TT\mutint\QTp\mutint\TC\mutint\Qr\mutint\AST\mutint\AST{+}\mathrm{EXT}
\mutint S^1_2$, \ and \
$\Rr\mutint\mathsf{WTC}^{-\varepsilon}\mutint\WT\mutint\mathsf{WQT}
\mutint\mathsf{WQT}^{*}$.
\end{theorem}

Here $\TC$ is Grzegorczyk's concatenation theory \cite{Grz05,GZ08}, shown mutually
interpretable with $\Qr$ through \cite{Gan09,Sve07,Vis09,Dam17}; $\AST$ is adjunctive
set theory, with axioms $\exists y\,\forall x\,\neg(x\in y)$ and
$\forall x,y\,\exists z\,\forall u\,(u\in z\leftrightarrow u\in x\vee u=y)$;
$\mathsf{WTC}^{-\varepsilon}$ is the weak concatenation theory of \cite{HH14}; and
$S^1_2$ is Buss's bounded arithmetic, with $S^1_2\mutint\Qr$ by \cite{FF13}. We use
Theorem~\ref{thm:chains} only through the two displayed chains, each member invoked
in the exact form cited.

\section{The theories $\WFm$ and $\Fm$}\label{sec:theories}

\subsection{Hereditarily finite multisets}
A \emph{multiset} over a set $X$ is a function $m\colon X\to\mathbb{N}$ with finite
support. The class $\HFM$ of \emph{hereditarily finite multisets} is
$\bigcup_{n}\mathcal{M}_n$, where $\mathcal{M}_0=\emptyset$ and
$\mathcal{M}_{n+1}$ is the set of finite multisets over $\mathcal{M}_n$; the
\emph{rank} of $m\in\HFM$ is the least $n$ with $m\in\mathcal{M}_{n+1}$. We write
$\emptyset$ for the empty multiset, $\enc{x}$ for the singleton multiset containing
$x$ once, and $x\U y$ for multiset union (multiplicities add). \emph{Immediate
membership} is defined by $u\in_1 v:\Leftrightarrow\exists w\,(v=w\U\enc{u})$, and
\emph{containment} $\sq$ by recursion on rank:
\[
x\sq t\quad:\Longleftrightarrow\quad x=t\ \text{ or }\ \exists u\,(u\in_1 t
\wedge x\sq u).
\]
The intended model is $\mm:=(\HFM;\emptyset,\enc{\cdot},\U,\sq)$, a structure for
the language $L=\{\emptyset,\enc{\cdot},\U,\sq\}$.

We record the algebraic facts used throughout: $(\HFM,\U,\emptyset)$ is the free
commutative monoid on the set of singletons $\{\enc{m}:m\in\HFM\}$. In particular it
is cancellative, and every element has a factorization into singletons, unique up to
order.

\subsection{The finitely axiomatized theory $\Fm$}
The axioms of $\Fm$ are listed in Figure~\ref{fig:F}. All variables are universally
quantified.

\begin{figure}[t]
\centering
\fbox{\begin{minipage}{0.96\textwidth}
\textbf{Axioms of $\Fm$}
\begin{itemize}\setlength\itemsep{1pt}
\item[$\mathrm{F}_1$:] $(x\U y)\U z=x\U(y\U z)$ \hfill(associativity)
\item[$\mathrm{F}_2$:] $x\U y=y\U x$ \hfill(commutativity)
\item[$\mathrm{F}_3$:] $x\U\emptyset=x$ \hfill(neutrality)
\item[$\mathrm{F}_4$:] $x\U z=y\U z\rightarrow x=y$ \hfill(cancellation)
\item[$\mathrm{F}_5$:] $\enc{x}\neq\emptyset$ \hfill(non-degeneracy)
\item[$\mathrm{F}_6$:] $x\U y=\emptyset\rightarrow x=\emptyset$ \hfill(positivity)
\item[$\mathrm{F}_7$:] $\enc{x}=\enc{y}\rightarrow x=y$ \hfill(injectivity)
\item[$\mathrm{F}_8$:] $u\U v=\enc{x}\U w\rightarrow
      \exists w'\,(u=\enc{x}\U w'\vee v=\enc{x}\U w')$ \hfill(Levi property)
\item[$\mathrm{F}_9$:] $x\sq t\leftrightarrow\bigl(x=t\vee\exists w\,\exists u\,
      (t=w\U\enc{u}\wedge x\sq u)\bigr)$ \hfill(containment recursion)
\end{itemize}
\end{minipage}}
\caption{The theory $\Fm$: $\mathrm{F}_1$--$\mathrm{F}_7$ are universal,
$\mathrm{F}_8$--$\mathrm{F}_9$ are $\forall\exists$.}\label{fig:F}
\end{figure}

\begin{proposition}[Soundness]\label{prop:soundF}
$\mm\models\Fm$.
\end{proposition}

\begin{proof}
$\mathrm{F}_1$--$\mathrm{F}_3$ are the monoid laws. $\mathrm{F}_4$ and
$\mathrm{F}_8$ hold in any free commutative monoid: cancellativity is immediate from
freeness, and if the singleton $\enc{x}$ occurs in the factorization of $u\U v$,
then by uniqueness of factorization it occurs in that of $u$ or of $v$, which yields
the witness $w'$. $\mathrm{F}_5$--$\mathrm{F}_7$ are immediate.
$\mathrm{F}_9$ restates the recursive definition of $\sq$, since
$\exists w\,(t=w\U\enc{u})$ expresses $u\in_1 t$.
\end{proof}

\subsection{Normal forms and the schematic theory $\WFm$}
Fix the length-lexicographic order $\preceq$ on closed $L$-terms, viewed as
strings over the symbols $\langle$, $\rangle$, $\U$, $\emptyset$, and define
the \emph{normal} closed $L$-terms
recursively: $\emptyset$ is normal; if $t_1\preceq\dots\preceq t_k$ $(k\ge 1)$ are
normal, then $\enc{t_1}\U(\enc{t_2}\U(\dots\U\enc{t_k}))$ is normal. Every closed
term $s$ has a \emph{normal form} $\nf(s)$, computed by recursively normalizing
arguments, flattening $\U$, deleting $\emptyset$-factors, and sorting.

\begin{lemma}[Normal Form Lemma]\label{lem:nf}
For closed $L$-terms $s,t$ the following are equivalent:
\textup{(i)} $\mm\models s=t$; \ \textup{(ii)} $\nf(s)\equiv\nf(t)$; \
\textup{(iii)} $\mathrm{F}_1\text{--}\mathrm{F}_3\vdash s=t$.
\end{lemma}

\begin{proof}
(iii)$\Rightarrow$(i) is Proposition~\ref{prop:soundF}. For
(i)$\Rightarrow$(ii), the value of a normal term determines the multiset of values
of its top-level constituents; by induction on rank, values of normal terms
determine the terms, and sortedness fixes the arrangement, so evaluation is
injective on normal terms. For (ii)$\Rightarrow$(iii), each normalization step
(re-association, transposition of adjacent factors, deletion of an
$\emptyset$-factor) is an instance of $\mathrm{F}_1$--$\mathrm{F}_3$ under the
congruence rules of equational logic, whence
$\mathrm{F}_1\text{--}\mathrm{F}_3\vdash s=\nf(s)$, and likewise for $t$.
\end{proof}

For a closed term $t$, let $\hull(t)$ be the finite set of normal terms $s$ with
$\mm\models s\sq t$. The theory $\WFm$ has the axioms
$\mathrm{F}_1,\mathrm{F}_2,\mathrm{F}_3$ together with the schemes
\begin{itemize}\setlength\itemsep{1pt}
\item[$\mathrm{W}_1$:] $\neg\,(s=t)$\quad for closed terms $s,t$ with
$\nf(s)\not\equiv\nf(t)$;
\item[$\mathrm{W}_2$:] $\forall x\,\bigl(x\sq t\leftrightarrow
\bigvee_{s\in\hull(t)}x=s\bigr)$\quad for each closed term $t$.
\end{itemize}
By Lemma~\ref{lem:nf}, together with the evident computability of $\nf$ and
$\hull$, the instance sets are decidable, so $\WFm$ is recursively
axiomatized; and $\mm\models\WFm$ by construction. Note that cancellation
$\mathrm{F}_4$ is deliberately \emph{not} an axiom of $\WFm$; this is what keeps
finite models of finite fragments available (Section~\ref{sec:thmA}).

\subsection{Basic consequences of $\Fm$}
\begin{fact}\label{fact:one}
$\Fm\vdash\enc{x}\U w\neq\emptyset$.
\end{fact}
\begin{proof}
If $\enc{x}\U w=\emptyset$ then $\mathrm{F}_6$ gives $\enc{x}=\emptyset$,
contradicting $\mathrm{F}_5$.
\end{proof}

\begin{fact}[Atomicity]\label{fact:atom}
$\Fm\vdash\enc{x}=u\U v\rightarrow
\bigl((u=\enc{x}\wedge v=\emptyset)\vee(u=\emptyset\wedge v=\enc{x})\bigr)$.
\end{fact}
\begin{proof}
From $u\U v=\enc{x}\U\emptyset$ (by $\mathrm{F}_3$), $\mathrm{F}_8$ yields $w'$
with $u=\enc{x}\U w'$ or $v=\enc{x}\U w'$. In the first case,
$\enc{x}\U\emptyset=\enc{x}\U(w'\U v)$ by $\mathrm{F}_1$, so $\mathrm{F}_4$ gives
$w'\U v=\emptyset$, and $\mathrm{F}_6$ (with $\mathrm{F}_2$) gives
$w'=v=\emptyset$, whence $u=\enc{x}$. The second case is symmetric.
\end{proof}

\begin{lemma}[Peeling Lemma]\label{lem:peel}
For each fixed $k\ge 1$,
\[
\Fm\vdash\ \enc{a_1}\U\dots\U\enc{a_k}=w\U\enc{c}\ \rightarrow\
\textstyle\bigvee_{i\le k}\,c=a_i .
\]
\end{lemma}
\begin{proof}
Induction on $k$ (in the metatheory; each instance is a single $\Fm$-derivation).
For $k=1$: by Fact~\ref{fact:atom} applied to $\enc{a_1}=w\U\enc{c}$, either
$\enc{c}=\emptyset$, contradicting $\mathrm{F}_5$, or $\enc{c}=\enc{a_1}$, whence
$c=a_1$ by $\mathrm{F}_7$. For $k+1$: apply $\mathrm{F}_8$ to
$\bigl(\enc{a_1}\U\dots\U\enc{a_k}\bigr)\U\enc{a_{k+1}}=\enc{c}\U w$. If the second
disjunct holds, $\enc{a_{k+1}}=\enc{c}\U w'$, and Fact~\ref{fact:atom} with
$\mathrm{F}_5,\mathrm{F}_7$ gives $c=a_{k+1}$. If the first holds,
$\enc{a_1}\U\dots\U\enc{a_k}=\enc{c}\U w'$, and the induction hypothesis applies.
\end{proof}

\begin{proposition}\label{prop:FextendsWF}
$\Fm$ proves every axiom of $\WFm$; hence $\Fm\supseteq\WFm$.
\end{proposition}

\begin{proof}
It suffices to treat the schemes. \emph{Scheme $\mathrm{W}_1$.} By
Lemma~\ref{lem:nf} we may assume $s,t$ normal and distinct, say
$s\equiv\enc{c_1}\U\dots\U\enc{c_k}$ and $t\equiv\enc{d_1}\U\dots\U\enc{d_l}$ with
$\{c_1,\dots,c_k\}\neq\{d_1,\dots,d_l\}$ as multisets of normal terms. We argue by
complete induction on the total number of symbols of the pair $(s,t)$; for
pairs of proper subterms this measure strictly decreases. If exactly one of
$k,l$ is $0$, then $s\neq t$ follows from Fact~\ref{fact:one}. Otherwise we
distinguish two cases. \emph{Case 1: some $c_i$ is syntactically $d_1$.}
Working in $\Fm$, suppose $s=t$. By $\mathrm{F}_1,\mathrm{F}_2$ both sides may
be rewritten with the factor $\enc{d_1}$ rightmost, and $\mathrm{F}_4$ cancels
it, yielding a provable equality of two normal terms whose top-level multisets
are $\{c_j:j\neq i\}$ and $\{d_2,\dots,d_l\}$. These are still distinct as
multisets, having arisen from distinct multisets by removal of one occurrence
of the same element, and the pair is strictly smaller; the induction
hypothesis refutes the equality --- contradiction. \emph{Case 2: no $c_i$ is
syntactically $d_1$.} Working in $\Fm$, suppose $s=t$. By
Lemma~\ref{lem:peel}, applied with the explicit decomposition
$t=\bigl(\enc{d_2}\U\dots\U\enc{d_l}\bigr)\U\enc{d_1}$ read against $s$, we
obtain the provable disjunction $\bigvee_{i\le k}d_1=c_i$; each disjunct is
refuted by the induction hypothesis applied to the strictly smaller pair
$(c_i,d_1)$ of distinct normal terms --- contradiction.
\emph{Scheme $\mathrm{W}_2$.} Fix a closed $t$; by Lemma~\ref{lem:nf} assume $t$
normal with top-level constituents $t_1,\dots,t_k$. For
($\Leftarrow$): for each $s\in\hull(t)$ there is a finite chain
$s\sq t_{j_1}\sq\dots\sq t$ along immediate memberships; each step is witnessed in
$\mathrm{F}_9$ by the explicit decomposition of the relevant normal term, so
$\Fm\vdash s\sq t$. For ($\Rightarrow$): argue by induction on the rank of $t$.
By $\mathrm{F}_9$, $x\sq t$ gives $x=t$ or witnesses $w,u$ with $t=w\U\enc{u}$ and
$x\sq u$. By Lemma~\ref{lem:peel}, provably $u=t_j$ for some $j\le k$ (if $k=0$,
Fact~\ref{fact:one} refutes the case). The induction hypothesis --- the instance of
$\mathrm{W}_2$ for $t_j$, already proved --- turns $x\sq t_j$ into the disjunction
over $\hull(t_j)\subseteq\hull(t)$, completing the proof.
\end{proof}

\begin{remark}
Proposition~\ref{prop:FextendsWF} is the exact analogue of the fact that $\TT$
extends $\WT$ \cite{KM20,Dam22}. In particular, the essential undecidability of
$\Fm$ already follows from Theorem~A and Fact~\ref{fact:transfer}; Theorem~B
strengthens this to mutual interpretability with $\Qr$.
\end{remark}

\section{The Pairing Lemma and a direct interpretation of $\TT$}\label{sec:pairing}

\subsection{Ordered pairs from bare multiplicity}
In strings and sequences, order is positional; in the tree theories, it is carried
by the argument places of the pairing constructor; in adjunctive set theory,
Kuratowski pairing exploits the idempotence of set formation. In $\HFM$ none of
these devices is available. Order is instead recovered from multiplicity alone:
define the \emph{pairing term}
\[
\pr{x}{y}\ :=\ \enc{x}\U\bigl(\enc{x}\U\enc{y}\bigr),
\]
the multiset $[x,x,y]$, in which the first component is marked by multiplicity two
and the second by multiplicity one. The next two statements make this precise.

\begin{corollary}[Elements of a pair]\label{cor:elements}
$\Fm\vdash\ \pr{u}{v}=w\U\enc{c}\ \rightarrow\ (c=u\vee c=v)$.
\end{corollary}

\begin{proof}
This is Lemma~\ref{lem:peel} with $k=3$ and $(a_1,a_2,a_3):=(u,u,v)$, modulo the
re-association $\enc{u}\U(\enc{u}\U\enc{v})=(\enc{u}\U\enc{u})\U\enc{v}$ provided
by $\mathrm{F}_1$.
\end{proof}

\begin{lemma}[Pairing Lemma]\label{lem:pairing}
$\Fm\vdash\ \pr{x}{y}=\pr{u}{v}\ \rightarrow\ (x=u\wedge y=v)$.
\end{lemma}

\begin{proof}
We reason in $\Fm$ from the hypothesis $\pr{x}{y}=\pr{u}{v}$. Throughout,
applications of $\mathrm{F}_4$ are preceded by the evident re-associations and
transpositions licensed by $\mathrm{F}_1,\mathrm{F}_2$, which we do not display.

Since $\pr{x}{y}=(\enc{x}\U\enc{y})\U\enc{x}$, the hypothesis exhibits
$\pr{u}{v}$ in the form $w\U\enc{x}$ with $w:=\enc{x}\U\enc{y}$, so
Corollary~\ref{cor:elements} yields the provable disjunction $x=u\vee x=v$. We
distinguish two cases.

\emph{Case A: $x=u$.} Substituting $u=x$ and cancelling $\enc{x}\U\enc{x}$ in one
application of $\mathrm{F}_4$ gives $\enc{y}=\enc{v}$, whence $y=v$ by
$\mathrm{F}_7$. Together with $x=u$ this is the claim.

\emph{Case B: $x\neq u$, hence $x=v$.} Substituting $v=x$ and cancelling one
occurrence of $\enc{x}$ gives
\[
\enc{x}\U\enc{y}\ =\ \enc{u}\U\enc{u}.
\]
By $\mathrm{F}_8$ there is $w'$ with $\enc{x}=\enc{u}\U w'$ or
$\enc{y}=\enc{u}\U w'$. In the first case, Fact~\ref{fact:atom} gives either
$\enc{u}=\enc{x}$, whence $x=u$ by $\mathrm{F}_7$, contradicting the case
hypothesis, or $\enc{u}=\emptyset$, contradicting $\mathrm{F}_5$. In the second
case, Fact~\ref{fact:atom} together with $\mathrm{F}_5$ gives $\enc{u}=\enc{y}$
and $w'=\emptyset$, whence $y=u$ by $\mathrm{F}_7$; substituting back yields
$\enc{x}\U\enc{u}=\enc{u}\U\enc{u}$, and cancellation gives $\enc{x}=\enc{u}$, so
$x=u$ by $\mathrm{F}_7$ --- again contradicting the case hypothesis. Case B is
therefore impossible, and the lemma follows from Case A.
\end{proof}

We record the exact bookkeeping: apart from the monoid laws
$\mathrm{F}_1$--$\mathrm{F}_3$, the proof uses $\mathrm{F}_4$, $\mathrm{F}_5$,
$\mathrm{F}_7$ and $\mathrm{F}_8$, the latter only through
Fact~\ref{fact:atom} and Corollary~\ref{cor:elements}. Every factorization
invoked has a fixed finite number of factors; no induction is required.

\subsection{The interpretation}
Let $\tau$ be the translation of $L_{\TT}$ into $L$ that assigns to the constant
$\bot$ the term $\emptyset$, to the function symbol $\langle\cdot,\cdot\rangle$
the term $\pr{x}{y}$, and to the relation symbol $\sq$ the atomic formula
$x\sq y$; equality is translated as equality and the domain is
$\delta(x):=(x=x)$, so $\tau$ is a candidate direct interpretation, and totality
and functionality of the translated function symbols are trivial.

\begin{theorem}\label{thm:TinF}
$\Fm$ proves the $\tau$-translations of the axioms
$\TT_1$--$\TT_4$. Hence $\tau$ is a direct interpretation of $\TT$ in $\Fm$, and
$\TT\intp\Fm$.
\end{theorem}

\begin{proof}
Four verifications.

$(\TT_1)^{\tau}$: $\forall x,y\;\neg(\pr{x}{y}=\emptyset)$. By $\mathrm{F}_1$,
$\pr{x}{y}=\enc{x}\U(\enc{x}\U\enc{y})$, and Fact~\ref{fact:one} gives
$\enc{x}\U w\neq\emptyset$.

$(\TT_2)^{\tau}$: $\forall x,y,z,w\;(\pr{x}{y}=\pr{z}{w}\rightarrow
x=z\wedge y=w)$. This is Lemma~\ref{lem:pairing}.

$(\TT_3)^{\tau}$: $\forall x\;(x\sq\emptyset\leftrightarrow x=\emptyset)$.
Instantiating $\mathrm{F}_9$ at $t:=\emptyset$ gives
\[
x\sq\emptyset\ \leftrightarrow\ x=\emptyset\ \vee\
\exists w\,\exists u\,(\emptyset=w\U\enc{u}\wedge x\sq u).
\]
The existential disjunct is refuted outright: from $\emptyset=w\U\enc{u}$ and
$\mathrm{F}_2$ we get $\enc{u}\U w=\emptyset$, contradicting
Fact~\ref{fact:one}. The equivalence therefore reduces to
$x\sq\emptyset\leftrightarrow x=\emptyset$.

$(\TT_4)^{\tau}$: $\forall x,u,v\;\bigl(x\sq\pr{u}{v}\leftrightarrow
x=\pr{u}{v}\vee x\sq u\vee x\sq v\bigr)$.

($\Rightarrow$) Suppose $x\sq\pr{u}{v}$. By $\mathrm{F}_9$, either
$x=\pr{u}{v}$, or there are $w,a$ with $\pr{u}{v}=w\U\enc{a}$ and $x\sq a$. In
the latter case Corollary~\ref{cor:elements} yields the provable disjunction
$a=u\vee a=v$; substituting equals in $x\sq a$ gives $x\sq u\vee x\sq v$. No
uniqueness of the witness $w$ is required.

($\Leftarrow$) If $x=\pr{u}{v}$, the left disjunct of $\mathrm{F}_9$ gives
$x\sq\pr{u}{v}$. If $x\sq u$, take $a:=u$ and $w:=\enc{u}\U\enc{v}$: by
$\mathrm{F}_1,\mathrm{F}_2$, $\pr{u}{v}=(\enc{u}\U\enc{v})\U\enc{u}=w\U\enc{a}$,
so the right disjunct of $\mathrm{F}_9$ gives $x\sq\pr{u}{v}$. If $x\sq v$, take
$a:=v$ and $w:=\enc{u}\U\enc{u}$: by $\mathrm{F}_1$,
$\pr{u}{v}=(\enc{u}\U\enc{u})\U\enc{v}=w\U\enc{a}$, and $\mathrm{F}_9$ applies as
before.
\end{proof}

\begin{corollary}\label{cor:QinF}
$\Qr\intp\Fm$. Consequently, $\Fm$ is essentially undecidable.
\end{corollary}

\begin{proof}
By Theorem~\ref{thm:QT}, $\Qr\intp\TT$; by Theorem~\ref{thm:TinF},
$\TT\intp\Fm$; transitivity of interpretability gives $\Qr\intp\Fm$. Since
$\Qr$ is essentially undecidable \cite{TMR53} and $\Fm$ is consistent
(Proposition~\ref{prop:soundF}), Fact~\ref{fact:transfer} applies.
\end{proof}

\section{Theorem A: $\WFm$ is mutually interpretable with $\Rr$}\label{sec:thmA}

The two directions are proved separately: $\Rr\intp\WFm$ by composing a
schema-to-schema interpretation of $\WT$ in $\WFm$ with
Theorem~\ref{thm:WTR}, and $\WFm\intp\Rr$ by Visser's criterion
(Theorem~\ref{thm:visser}), using finite truncations of the intended model.

\subsection{Coding trees as multisets}
Define the coding $c$ from variable-free $L_{\TT}$-terms to closed $L$-terms by
\[
c(\bot):=\emptyset,\qquad
c(\langle s,t\rangle):=\pr{c(s)}{c(t)}.
\]
Thus $c(r)$ is precisely the $\tau$-translation of the closed term $r$ under the
translation $\tau$ of Section~\ref{sec:pairing}.

\begin{lemma}[Coding Lemma]\label{lem:coding}
The evaluation of codes in $\mm$ is injective: for distinct variable-free
$L_{\TT}$-terms $s,t$ we have $\mm\not\models c(s)=c(t)$, and hence
$\nf(c(s))\not\equiv\nf(c(t))$.
\end{lemma}

\begin{proof}
By induction on the maximum of the heights of $s,t$; the final claim then follows
from Lemma~\ref{lem:nf}. Write $\mathrm{val}(\cdot)$ for evaluation in $\mm$.
First, $\mathrm{val}(c(\langle s,t\rangle))$ is a nonempty multiset while
$\mathrm{val}(c(\bot))=\emptyset$, so terms of different outermost shape have
distinct values. Suppose now
$\mathrm{val}(c(\langle s_1,t_1\rangle))=\mathrm{val}(c(\langle s_2,t_2\rangle))$,
i.e.\ $[a,a,b]=[a',a',b']$ as multisets, where $a:=\mathrm{val}(c(s_1))$,
$b:=\mathrm{val}(c(t_1))$, $a':=\mathrm{val}(c(s_2))$, $b':=\mathrm{val}(c(t_2))$.
If $a=b$, the left-hand side is $[a,a,a]$, so every element of the right-hand
side equals $a$; in particular $a'=a$ and $b'=a=b$. If $a\neq b$, then the
left-hand side has exactly two distinct elements, with multiplicities $2$ and
$1$; hence $a'\neq b'$ (otherwise the right-hand side would be $[a',a',a']$), and
matching multiplicities forces $a=a'$ and $b=b'$. In either case
$\mathrm{val}(c(s_1))=\mathrm{val}(c(s_2))$ and
$\mathrm{val}(c(t_1))=\mathrm{val}(c(t_2))$, so the induction hypothesis yields
$s_1=s_2$ and $t_1=t_2$.
\end{proof}

\begin{lemma}[Hull Lemma]\label{lem:hullcode}
For every variable-free $L_{\TT}$-term $t$,
\[
\hull(c(t))\ =\ \{\nf(c(s)):s\in S(t)\},
\]
and the map $s\mapsto\nf(c(s))$ is a bijection from $S(t)$ onto $\hull(c(t))$.
\end{lemma}

\begin{proof}
Injectivity of $s\mapsto\nf(c(s))$ is Lemma~\ref{lem:coding}; it remains to
compute the hull, by induction on $t$. For $t=\bot$: no $m\in\HFM$ satisfies
$\emptyset=w\U\enc{u}$, so $\emptyset$ has no immediate members and
$x\sq\emptyset$ holds only for $x=\emptyset$; thus
$\hull(\emptyset)=\{\emptyset\}=\{\nf(c(\bot))\}$. For $t=\langle t_1,t_2\rangle$:
the value of $c(t)$ is $[a,a,b]$ with $a:=\mathrm{val}(c(t_1))$,
$b:=\mathrm{val}(c(t_2))$, whose immediate members are exactly the values $a$ and
$b$ (multiplicity does not affect immediate membership). By the recursive
definition of $\sq$, $x\sq\mathrm{val}(c(t))$ iff $x=\mathrm{val}(c(t))$ or
$x\sq a$ or $x\sq b$; by the induction hypothesis the latter two cases range
exactly over the values of the codes of $S(t_1)$ and $S(t_2)$. Since
$S(\langle t_1,t_2\rangle)=\{\langle t_1,t_2\rangle\}\cup S(t_1)\cup S(t_2)$, the
claim follows.
\end{proof}

\subsection{The schema-to-schema interpretation}

\begin{proposition}\label{prop:WTinWF}
The translation $\tau$ of Section~\ref{sec:pairing} is a direct interpretation of
$\WT$ in $\WFm$; hence $\WT\intp\WFm$.
\end{proposition}

\begin{proof}
Totality and functionality are trivial as before, since function symbols are
translated by terms. It remains to show that $\WFm$ proves the translation of
every instance of the two schemes.

\emph{Scheme $\mathrm{WT}_1$.} An instance is $\neg(s=t)$ for distinct
variable-free $L_{\TT}$-terms; its translation is $\neg(c(s)=c(t))$. By
Lemma~\ref{lem:coding}, $\nf(c(s))\not\equiv\nf(c(t))$, so $\neg(c(s)=c(t))$ is
an instance of $\mathrm{W}_1$, hence an axiom of $\WFm$.

\emph{Scheme $\mathrm{WT}_2$.} An instance is
$\forall x\,(x\sq t\leftrightarrow\bigvee_{s\in S(t)}x=s)$ for a variable-free
$t$; its translation is
$\forall x\,(x\sq c(t)\leftrightarrow\bigvee_{s\in S(t)}x=c(s))$. The instance of
$\mathrm{W}_2$ for the closed term $c(t)$ reads
$\forall x\,(x\sq c(t)\leftrightarrow\bigvee_{r\in\hull(c(t))}x=r)$. By
Lemma~\ref{lem:hullcode}, $r$ ranges exactly over the terms $\nf(c(s))$ for
$s\in S(t)$, and by Lemma~\ref{lem:nf}(iii),
$\mathrm{F}_1\text{--}\mathrm{F}_3\vdash c(s)=\nf(c(s))$ for each $s$. The two
disjunctions are therefore provably equivalent in $\WFm$, and the translated
instance follows.
\end{proof}

\begin{corollary}\label{cor:RinWF}
$\Rr\intp\WFm$.
\end{corollary}

\begin{proof}
$\Rr\intp\WT$ by Theorem~\ref{thm:WTR}, and $\WT\intp\WFm$ by
Proposition~\ref{prop:WTinWF}; compose.
\end{proof}

\subsection{Local finite satisfiability}
Define the \emph{weight} of $m\in\HFM$ by rank recursion:
$W(\emptyset):=0$ and $W([m_1,\dots,m_k]):=k+\sum_{i\le k}W(m_i)$, elements
listed with multiplicity. Correspondingly, for closed $L$-terms put
$w(\emptyset):=0$, $w(\enc{t}):=w(t)+1$, $w(s\U t):=w(s)+w(t)$; a straightforward
induction gives $W(\mathrm{val}(t))=w(t)$, and $w(t')\le w(t)$ whenever $t'$ is a
subterm of $t$. Moreover, if $x\sq m$ then $W(x)\le W(m)$: for $x\neq m$ there
is $u\in_1 m$ with $x\sq u$, and $m=w\U\enc{u}$ gives $W(u)<W(m)$, so the claim
follows by rank induction. Finally, for each $B$ the set
$\{m\in\HFM:W(m)\le B\}$ is finite, by induction on $B$: such an $m$ has at most
$B$ top-level members, each of weight $<B$.

\begin{lemma}\label{lem:MB}
Every finite subset $\Sigma_0$ of the axioms of $\WFm$ has a finite model.
\end{lemma}

\begin{proof}
Enlarging $\Sigma_0$, we may assume it contains
$\mathrm{F}_1,\mathrm{F}_2,\mathrm{F}_3$ together with finitely many instances
of $\mathrm{W}_1$ and $\mathrm{W}_2$. Let $B\ge 1$ bound the weights $w(t)$ of
all closed terms occurring in $\Sigma_0$. Define the finite structure $M_B$:
universe $\{m\in\HFM:W(m)\le B\}\cup\{\infty\}$, where $\infty$ is a fresh
object; operations
\[
\begin{aligned}
\enc{x}^{M_B}&:=\begin{cases}\enc{x}&\text{if }x\text{ standard, }W(x)+1\le B,\\
\infty&\text{otherwise;}\end{cases}\\[3pt]
x\U^{M_B}y&:=\begin{cases}x\U y&\text{if }x,y\text{ standard, }W(x)+W(y)\le B,\\
\infty&\text{otherwise.}\end{cases}
\end{aligned}
\]
$\emptyset^{M_B}:=\emptyset$; and $\sq^{M_B}$ is the restriction of $\sq$ to the
standard part, with $\infty$ unrelated to anything.

\emph{$\mathrm{F}_1$}: if $x,y,z$ are standard and $W(x)+W(y)+W(z)\le B$, both
sides evaluate to the true $x\U y\U z$, since every partial sum is bounded by the
total; otherwise both sides evaluate to $\infty$ --- if, say,
$W(x)+W(y)>B$, then the left side is $\infty\U^{M_B}z=\infty$, while on the
right either $y\U^{M_B}z=\infty$ already, or the outer union has total weight
$>B$. \emph{$\mathrm{F}_2$}: the definition is symmetric.
\emph{$\mathrm{F}_3$}: $W(x)+W(\emptyset)=W(x)\le B$ for standard $x$, and
$\infty\U^{M_B}\emptyset=\infty$.

\emph{Instances of $\mathrm{W}_1$ in $\Sigma_0$}: for a closed term $t$ occurring
in $\Sigma_0$, every subterm $t'$ has $w(t')\le w(t)\le B$, so the evaluation of
$t$ in $M_B$ never truncates and coincides with its evaluation in $\mm$. Distinct
normal forms thus receive distinct values by Lemma~\ref{lem:nf}.

\emph{Instances of $\mathrm{W}_2$ in $\Sigma_0$}: fix such an instance, for the
closed term $t$. Every $s\in\hull(t)$ satisfies
$W(\mathrm{val}(s))\le W(\mathrm{val}(t))\le B$, so all terms occurring in the
instance evaluate standardly, as above. For standard $x$: $x\sq^{M_B}
\mathrm{val}(t)$ iff $x\sq\mathrm{val}(t)$ in $\mm$, iff $x$ is the value of some
$s\in\hull(t)$, by the definition of $\hull$ and the bijectivity of evaluation on
normal terms (Lemma~\ref{lem:nf}). For $x=\infty$: the left side is false by
definition of $\sq^{M_B}$, and each disjunct $\infty=\mathrm{val}(s)$ is false
since $\mathrm{val}(s)$ is standard.
\end{proof}

\begin{remark}
In $M_B$ cancellation fails for $B\ge 2$: for $z$ standard of weight $B$ and distinct
standard $x,y$ of positive weight, $x\U^{M_B}z=\infty=y\U^{M_B}z$. The omission
of $\mathrm{F}_4$ from $\WFm$ is thus essential to this construction, as
announced in Section~\ref{sec:theories}.
\end{remark}

\begin{proposition}\label{prop:WFinR}
$\WFm\intp\Rr$.
\end{proposition}

\begin{proof}
$\WFm$ is recursively axiomatized (Section~\ref{sec:theories}) and, by
Lemma~\ref{lem:MB}, locally finitely satisfiable; apply
Theorem~\ref{thm:visser}.
\end{proof}

\subsection{Theorem A and closing corollaries}

\begin{theorem}[Theorem A]\label{thm:A}
$\WFm\mutint\Rr$.
\end{theorem}

\begin{proof}
Corollary~\ref{cor:RinWF} and Proposition~\ref{prop:WFinR}.
\end{proof}

\begin{corollary}[The $\Rr$-cluster]\label{cor:cluster}
$\WFm$ is mutually interpretable with each of $\WT$,
$\mathsf{WTC}^{-\varepsilon}$, $\mathsf{WQT}$ and $\mathsf{WQT}^{*}$.
\end{corollary}

\begin{proof}
Theorem~\ref{thm:A} and the second chain of Theorem~\ref{thm:chains}.
\end{proof}

\begin{corollary}\label{cor:closing}
\textup{(i)} $\WFm$ is not finitely axiomatizable.
\textup{(ii)} $\Fm$ is not interpretable in $\Rr$.
\end{corollary}

\begin{proof}
Every model of $\WFm$ is infinite: the closed terms
$\emptyset,\enc{\emptyset},\enc{\enc{\emptyset}},\dots$ have pairwise distinct
normal forms, and the corresponding $\mathrm{W}_1$ instances force their values
to be pairwise distinct.

(i) Suppose $\Sigma$ were a finite axiomatization of $\WFm$. Then $\Sigma$ and
$\WFm$ have the same models and the same theorems, so $\Sigma\intp\Rr$ by
Theorem~\ref{thm:A}, and Theorem~\ref{thm:visser} makes $\Sigma$ locally
finitely satisfiable; being finite, $\Sigma$ would itself have a finite model,
which is a model of $\WFm$ --- contradiction.

(ii) By Proposition~\ref{prop:FextendsWF}, every model of $\Fm$ is a model of
$\WFm$, hence infinite. If $\Fm\intp\Rr$, then Theorem~\ref{thm:visser} would
give the finitely axiomatized $\Fm$ a finite model --- contradiction.
\end{proof}

\begin{remark}
Corollary~\ref{cor:closing}(ii) parallels the observation of \cite{KM20} that
$\TT$, being finitely axiomatized with only infinite models, is not
interpretable in $\Rr$. The pair $(\WFm,\Fm)$ thus reproduces, over multisets,
the exact interpretability configuration of the pair $(\WT,\TT)$ over trees:
the schematic theory sits in the cluster of $\Rr$, the finitely axiomatized
theory strictly above it.
\end{remark}

\section{Theorem B: $\Fm$ is mutually interpretable with $\Qr$}\label{sec:thmB}

By Corollary~\ref{cor:QinF}, $\Qr\intp\Fm$. It remains to prove $\Fm\intp\Qr$;
since $S^1_2\mutint\Qr$ (Theorem~\ref{thm:chains}), it suffices to construct an
interpretation of $\Fm$ in Buss's bounded arithmetic $S^1_2$ \cite{Bus86}. The
interpretation arithmetizes the normal-form calculus of
Section~\ref{sec:theories}: hereditarily finite multisets are represented by
canonical words over the bracket alphabet, union by sorted merge, and
containment by occurrence as a balanced segment.

Throughout this section we work in $S^1_2$ and use freely its $\Sigma^b_1$-definable
functions, the coding of bounded sequences, and induction $\mathrm{PIND}$ on
$\Sigma^b_1$-formulas \cite{Bus86,HP93}. All syntactic predicates and operations
introduced below are polynomial-time computable; we identify them with fixed
$\Delta^b_1$-definitions for which $S^1_2$ proves the elementary properties stated.

\subsection{Words, blocks, and canonicity}
Words over the alphabet $\{\langle,\rangle\}$ are identified with their dyadic
codes; $\varepsilon$ denotes the empty word, with code $0$, and $|w|$ the length
of $w$. The \emph{depth sequence} of $w$ assigns to each $i\le|w|$ the number of
occurrences of $\langle$ minus the number of occurrences of $\rangle$ among the
first $i$ letters; it is a $\Sigma^b_1$-definable function of $w$, coded as a
bounded sequence. A word $w$ is \emph{balanced} if its depth sequence is
nonnegative and ends at $0$. A pair $(i,j)$ with $1\le i\le j\le|w|$ is a
\emph{block} of $w$, written $\mathrm{Blk}(w,i,j)$, if the letter at $i$ is
$\langle$, the letter at $j$ is its matching $\rangle$ (the first later position
where the depth returns to its value before $i$); we write $w[i..j]$ for the
corresponding segment. A block is \emph{top-level} if the depth before $i$
is~$0$.

\begin{lemma}[Block structure]\label{lem:blocks}
$S^1_2$ proves: \textup{(i)} any two blocks of a balanced word are disjoint or
nested; \textup{(ii)} every balanced $w\neq\varepsilon$ has a unique
decomposition $w=b_1{}^{\frown}\dots{}^{\frown}b_k$ into its top-level blocks,
listed in positional order (we write $\mathrm{TL}(w)$ for this sequence);
\textup{(iii)} a block of $w$ is either top-level or contained in the interior
$u_m$ of exactly one top-level block $b_m=\langle u_m\rangle$, and it is then a
block of the word $u_m$ under the evident index shift; conversely every block of
$u_m$ is, under the same shift, a block of $w$.
\end{lemma}

\begin{proof}
(i) If two blocks overlap without nesting, the matching condition on depths
fails for one of them. (ii) The top-level blocks begin exactly at the positions
of depth $0$ carrying $\langle$; matching is unique, and the segments partition
$w$ by (i). (iii) By (i), a non-top-level block is nested in some top-level
$b_m$, unique by disjointness. If it began at the initial $\langle$ of $b_m$,
uniqueness of the matching bracket would make it equal to $b_m$; hence it lies
strictly within the interior, and the depth sequence of $u_m$ is that of $w$
shifted by one, so blocks correspond under the shift. All statements are
$\Delta^b_1$-properties of the depth sequence, provable by $\mathrm{PIND}$.
\end{proof}

Let $\le_{\mathrm{ll}}$ denote the length-lexicographic order on words: $u
\le_{\mathrm{ll}} v$ iff $|u|<|v|$, or $|u|=|v|$ and $u$ is lexicographically
$\le v$. $S^1_2$ proves $\le_{\mathrm{ll}}$ is a linear order. A sequence of
words is \emph{sorted} if it is nondecreasing in $\le_{\mathrm{ll}}$. Define
$\mathrm{Can}(w)$ ($w$ is \emph{canonical}) by: $w$ is balanced, and for every
block $(i,j)$ of $w$ (including the improper block when $w$ itself is treated at
top level), the sequence of top-level blocks of its interior is sorted.
Equivalently: $w=\varepsilon$, or $\mathrm{TL}(w)$ is sorted and the interior of
each entry is canonical. $\mathrm{Can}$ is $\Delta^b_1$.

\begin{lemma}[Rigidity]\label{lem:rigid}
$S^1_2$ proves: if $L,L'$ are sorted sequences of words and $\pi$ is a bijection
of positions with $L'(\pi(i))=L(i)$ for all $i$, then $L=L'$.
\end{lemma}

\begin{proof}
$\mathrm{PIND}$ on the common length. The first entry of $L$ occurs in $L'$, so
$L'(1)\le_{\mathrm{ll}}L(1)$ by sortedness; symmetrically
$L(1)\le_{\mathrm{ll}}L'(1)$, and antisymmetry gives $L(1)=L'(1)$. Composing
$\pi$ with a transposition, we may assume $\pi(1)=1$; delete the first entries
and apply the induction hypothesis.
\end{proof}

\subsection{The operations}
Define: $\mathrm{Wr}(x):=\langle x\rangle$ (\emph{wrapping}); for a block word
$b=\langle u\rangle$ with $\mathrm{Can}(u)$, $\mathrm{Ins}(w,b)$ is the word
obtained by inserting $b$ into $\mathrm{TL}(w)$ at the first position whose
entry is $\ge_{\mathrm{ll}}b$ (at the end if none); $\mathrm{Rem}(w,b)$ is the
word obtained by deleting the first entry of $\mathrm{TL}(w)$ equal to $b$
(defined when one exists; we write $\mathrm{Occ}(w,b)$ for existence);
$\mathrm{Mrg}(x,y)$ is the result of inserting into $x$, successively in
positional order, the entries of $\mathrm{TL}(y)$. The iteration is coded by
the bounded sequence of intermediate words, each of length $\le|x|+|y|$, so
$\mathrm{Mrg}$ is $\Sigma^b_1$-definable with $S^1_2$-provable totality and
uniqueness on canonical arguments.

\begin{lemma}[Bookkeeping]\label{lem:book}
$S^1_2$ proves, for canonical $w,x,y$ and blocks $a,b$ with canonical
interiors:
\textup{(i)} $\mathrm{Can}(\mathrm{Wr}(x))$,
$\mathrm{Can}(\mathrm{Ins}(w,b))$, $\mathrm{Can}(\mathrm{Mrg}(x,y))$, and
$\mathrm{Can}(\mathrm{Rem}(w,b))$ when defined;
\textup{(ii)} $\mathrm{Rem}(\mathrm{Ins}(w,b),b)=w$;
\textup{(iii)} if $\mathrm{Occ}(w,b)$ then
$\mathrm{Ins}(\mathrm{Rem}(w,b),b)=w$;
\textup{(iv)} $|\mathrm{Ins}(w,b)|=|w|+|b|$, hence
$|\mathrm{Mrg}(x,y)|=|x|+|y|$;
\textup{(v)} $\mathrm{TL}(\mathrm{Mrg}(x,y))$ is a permutation of
$\mathrm{TL}(x){}^{\frown}\mathrm{TL}(y)$, with a $\Sigma^b_1$-witnessing
bijection; in particular $\mathrm{Occ}(\mathrm{Mrg}(x,y),b)$ iff
$\mathrm{Occ}(x,b)$ or $\mathrm{Occ}(y,b)$.
\end{lemma}

\begin{proof}
(i) Insertion preserves sortedness by choice of position and leaves interiors
untouched; wrapping produces a single-block word. (ii) The inserted entry is
the first occurrence of $b$: entries equal to $b$ that precede the insertion
point would contradict its minimality, so deletion of the first occurrence
removes the inserted entry --- and if equal entries occur, deleting any
occurrence yields the same sequence. (iii) Symmetrically: reinsertion of $b$
lands at the position of the deleted first occurrence, by sortedness. (iv)
Immediate from the definitions, by $\mathrm{PIND}$ along the iteration for
$\mathrm{Mrg}$. (v) Each insertion adds exactly its argument to the entry pool
and preserves the rest; compose the witnessing bijections along the iteration.
\end{proof}

\subsection{The translation and the axioms $\mathrm{F}_1$--$\mathrm{F}_8$}
Let $\rho$ be the translation of $L$ into the language of $S^1_2$ with domain
$\delta(x):=\mathrm{Can}(x)$, equality as identity, $\emptyset\mapsto
\varepsilon$, $\enc{\cdot}\mapsto\mathrm{Wr}$, $\U\mapsto\mathrm{Mrg}$, and
$x\sq y$ translated by the formula $\mathrm{Sub}(x,y)$ of
Section~\ref{subsec:contain}. The domain is nonempty
($\mathrm{Can}(\varepsilon)$) and provably
closed under the operations (Lemma~\ref{lem:book}(i)).

\begin{proposition}\label{prop:F18}
$S^1_2$ proves the $\rho$-translations of
$\mathrm{F}_1$--$\mathrm{F}_8$.
\end{proposition}

\begin{proof}
$\mathrm{F}_1,\mathrm{F}_2$: by Lemma~\ref{lem:book}(v), both sides of each
identity are canonical words whose top-level sequences are sorted arrangements
of pools related by a bijection (composition of the witnessing bijections with
the evident regrouping); Lemma~\ref{lem:rigid} applied to the two sorted
sequences gives equality of the sequences, hence of the words by
Lemma~\ref{lem:blocks}(ii).

$\mathrm{F}_3$: $\mathrm{TL}(\varepsilon)$ is empty, so the iteration is empty
and $\mathrm{Mrg}(x,\varepsilon)=x$.

$\mathrm{F}_4$: first, single-block cancellation: if
$\mathrm{Ins}(u,b)=\mathrm{Ins}(v,b)$, then applying $\mathrm{Rem}(\cdot\,,b)$
to both sides and invoking Lemma~\ref{lem:book}(ii) gives $u=v$. Now argue by
$\mathrm{PIND}$ on the number $k$ of top-level blocks of $z$, with $x,y$ as
parameters; the induction formula is $\Delta^b_1$. For $k=0$, $z=\varepsilon$
and $\mathrm{F}_3$ applies. For $k>0$, write $z^-$ for the word of the first $k-1$
top-level blocks of $z$ and $b$ for the last; $z^-$ is canonical (a sorted
prefix with untouched interiors) and, by the definition of $\mathrm{Mrg}$ as
positional iteration, $\mathrm{Mrg}(x,z)=\mathrm{Ins}(\mathrm{Mrg}(x,z^-),b)$.
The hypothesis $\mathrm{Mrg}(x,z)=\mathrm{Mrg}(y,z)$ thus yields, by
single-block cancellation, $\mathrm{Mrg}(x,z^-)=\mathrm{Mrg}(y,z^-)$, and the
induction hypothesis gives $x=y$.

$\mathrm{F}_5$: $|\mathrm{Wr}(x)|=|x|+2>0=|\varepsilon|$.

$\mathrm{F}_6$: if $\mathrm{Mrg}(x,y)=\varepsilon$ then
Lemma~\ref{lem:book}(iv) gives $|x|+|y|=0$, so $x=\varepsilon$.

$\mathrm{F}_7$: equality of the words $\langle x\rangle=\langle y\rangle$
gives equality of the interiors.

$\mathrm{F}_8$: suppose
$\mathrm{Mrg}(u,v)=\mathrm{Mrg}(w,\mathrm{Wr}(c))=
\mathrm{Ins}(w,\mathrm{Wr}(c))$. The inserted entry is a top-level block of the
right-hand side, so $\mathrm{Occ}(\mathrm{Mrg}(u,v),\mathrm{Wr}(c))$, and
Lemma~\ref{lem:book}(v) gives $\mathrm{Occ}(u,\mathrm{Wr}(c))$ or
$\mathrm{Occ}(v,\mathrm{Wr}(c))$. In the first case put
$w':=\mathrm{Rem}(u,\mathrm{Wr}(c))$: Lemma~\ref{lem:book}(iii) gives
$u=\mathrm{Ins}(w',\mathrm{Wr}(c))=\mathrm{Mrg}(w',\mathrm{Wr}(c))$, which by
the already established $\mathrm{F}_2^{\rho}$ equals
$\mathrm{Mrg}(\mathrm{Wr}(c),w')$, the translation of $\enc{c}\U w'$. The
second case is symmetric.
\end{proof}

\subsection{Containment and the axiom $\mathrm{F}_9$}\label{subsec:contain}
Define
\[
\mathrm{Sub}(x,t)\ :\Longleftrightarrow\ x=t\ \vee\ \exists i,j\le|t|\,
\bigl(\mathrm{Blk}(t,i,j)\wedge t[i..j]=\mathrm{Wr}(x)\bigr).
\]

\begin{proposition}\label{prop:F9}
$S^1_2$ proves the $\rho$-translation of $\mathrm{F}_9$: for canonical $x,t$,
\[
\mathrm{Sub}(x,t)\ \leftrightarrow\ x=t\ \vee\ \exists w,u\,
\bigl(\mathrm{Can}(w)\wedge\mathrm{Can}(u)\wedge
t=\mathrm{Mrg}(w,\mathrm{Wr}(u))\wedge\mathrm{Sub}(x,u)\bigr).
\]
\end{proposition}

\begin{proof}
($\Leftarrow$) Assume $t=\mathrm{Mrg}(w,\mathrm{Wr}(u))=
\mathrm{Ins}(w,\mathrm{Wr}(u))$ and $\mathrm{Sub}(x,u)$. The inserted entry
$\mathrm{Wr}(u)$ is a top-level block of $t$, at positions $(p,q)$ say. If
$x=u$, then $t[p..q]=\mathrm{Wr}(x)$ witnesses $\mathrm{Sub}(x,t)$. If
$x\neq u$, then $\mathrm{Sub}(x,u)$ provides a block $(i',j')$ of $u$ with
$u[i'..j']=\mathrm{Wr}(x)$; by Lemma~\ref{lem:blocks}(iii) the shifted pair
$(p+i',p+j')$ is a block of $t$ with the same segment, witnessing
$\mathrm{Sub}(x,t)$.

($\Rightarrow$) Assume $\mathrm{Sub}(x,t)$ with $x\neq t$, witnessed by a block
$(i,j)$ of $t$ with $t[i..j]=\mathrm{Wr}(x)$. If $(i,j)$ is top-level, put
$u:=x$ and $w:=\mathrm{Rem}(t,\mathrm{Wr}(x))$: Lemma~\ref{lem:book}(iii) and
$\mathrm{F}_2^{\rho}$ give $t=\mathrm{Mrg}(w,\mathrm{Wr}(u))$, and
$\mathrm{Sub}(x,u)$ holds by the left disjunct $x=u$. If $(i,j)$ is not
top-level, Lemma~\ref{lem:blocks}(iii) places it, under an index shift, as a
block of the interior $u_m$ of a unique top-level block $\langle u_m\rangle$ of
$t$; thus $\mathrm{Sub}(x,u_m)$. Putting $u:=u_m$ and
$w:=\mathrm{Rem}(t,\langle u_m\rangle)$, Lemma~\ref{lem:book}(iii) gives
$t=\mathrm{Mrg}(w,\mathrm{Wr}(u_m))$ as before. No uniqueness of the witness
$w$ is used at any point.
\end{proof}

\subsection{Assembly}

\begin{theorem}[Theorem B]\label{thm:B}
$\Fm\mutint\Qr$. In particular, $\Fm$ is essentially undecidable.
\end{theorem}

\begin{proof}
By Propositions~\ref{prop:F18} and~\ref{prop:F9}, $\rho$ is an interpretation
of $\Fm$ in $S^1_2$, so $\Fm\intp S^1_2$; with $S^1_2\mutint\Qr$
(Theorem~\ref{thm:chains}) and transitivity, $\Fm\intp\Qr$. The converse is
Corollary~\ref{cor:QinF}, whose proof also delivers essential undecidability.
\end{proof}

\begin{corollary}[The $\Qr$-cluster]\label{cor:Qcluster}
$\Fm$ is mutually interpretable with each of $\TT$, $\QTp$, $\TC$, $\AST$,
$\AST+\mathrm{EXT}$, and $S^1_2$.
\end{corollary}

\begin{proof}
Theorem~\ref{thm:B} and the first chain of Theorem~\ref{thm:chains}.
\end{proof}

\begin{remark}
Theorems A and B, with Corollary~\ref{cor:closing}, place the pair
$(\WFm,\Fm)$ exactly as announced: the schematic theory of hereditarily finite
multisets lies in the mutual-inter\-pretability class of $\Rr$, the finitely
axiomatized theory in that of $\Qr$, and the two classes are separated by
Corollary~\ref{cor:closing}(ii). Hereditarily finite multisets thereby join
numbers, strings, trees, sets, and sequences in both clusters.
\end{remark}

\section{Minimality relative to the axiomatization}\label{sec:minimality}

By \cite{MPV24} there is no minimal recursively enumerable essentially
undecidable theory with respect to interpretability; minimality claims are
therefore meaningful only relative to a fixed axiomatization. The reference
pattern is Theorem~11 of \cite[Ch.~II]{TMR53}: every theory axiomatized by a
proper subset of the seven axioms of $\Qr$ has a consistent decidable
extension. Analogous results were obtained for concatenation theories by
Murwanashyaka \cite{Mur22} and by Higuchi and Horihata \cite{HH14}, and, for
$\Rr$, related results of Cobham are reported in \cite{JS83} (cf.\
\cite{MPV24}). This section establishes the corresponding profile for $\Fm$:
each of the structural axioms $\mathrm{F}_4$--$\mathrm{F}_8$ is certified by a
decidable witness (Theorem~\ref{thm:indep}); the containment axiom
$\mathrm{F}_9$ is conservative over $\mathrm{F}_1$--$\mathrm{F}_8$, hence
carries no algebraic content requiring certification
(Proposition~\ref{prop:cons}); and $\mathrm{F}_1$--$\mathrm{F}_3$ delimit the
subject matter (Remark~\ref{rem:frame}). For $i\in\{4,\dots,8\}$ we write
$\Fm{-}\mathrm{F}_i$ for the theory axiomatized by the remaining eight axioms
of Figure~\ref{fig:F}.

\subsection{Independence of the structural axioms}

\begin{figure}[t]
\centering
\fbox{\begin{minipage}{0.96\textwidth}
\centering\medskip
\begin{tabular}{c@{\quad}c@{\quad}c@{\quad}c@{\quad}c@{\quad}c@{\quad}c}
 & domain & $\emptyset$ & $\enc{x}$ & $x\U y$ & $x\sq t$ iff & fails\\[2pt]
$M_4$ & $\mathbb{N}$ & $0$ & $x+1$ & $\max(x,y)$ & $x\le t$ & $\mathrm{F}_4$\\
$M_5$ & $\{0\}$ & $0$ & $0$ & $0$ & $x=t=0$ & $\mathrm{F}_5$\\
$M_6$ & $\mathbb{Z}$ & $0$ & $2x+1$ & $x+y$ & always & $\mathrm{F}_6$\\
$M_7$ & $\mathbb{N}$ & $0$ & $1$ & $x+y$ & $t\ge 1\vee(x{=}0\wedge t{=}0)$
 & $\mathrm{F}_7$\\
$M_8$ & $\mathbb{N}$ & $0$ & $x+1$ & $x+y$ & $x\le t$ & $\mathrm{F}_8$\\[2pt]
\end{tabular}
\medskip
\end{minipage}}
\caption{The five witnesses. Each column lists the interpretation of the
$L$-symbols.}\label{fig:witnesses}
\end{figure}

\begin{lemma}[Decidability transfer]\label{lem:transfer}
Let $N$ be a structure whose domain and whose interpretations of all symbols
are first-order definable without parameters in a structure $M$, or, more
generally, let $N$ be parameter-free interpretable in $M$. If the complete
theory of $M$ is decidable, then so is the complete theory of $N$.
\end{lemma}

\begin{proof}
The defining formulas induce an effective translation $\varphi\mapsto
\varphi^{*}$ with $N\models\varphi$ iff $M\models\varphi^{*}$; decide
$\varphi^{*}$ in $\mathrm{Th}(M)$.
\end{proof}

\begin{theorem}[Independence of $\mathrm{F}_4$--$\mathrm{F}_8$]\label{thm:indep}
For each $i\in\{4,\dots,8\}$, the structure $M_i$ of
Figure~\ref{fig:witnesses} satisfies every axiom of $\Fm$ except
$\mathrm{F}_i$, refutes $\mathrm{F}_i$, and has a decidable complete
first-order theory. Consequently: \textup{(i)} $\mathrm{F}_i$ is not derivable
in $\Fm{-}\mathrm{F}_i$; \textup{(ii)} $\mathrm{Th}(M_i)$ is a consistent
decidable complete extension of $\Fm{-}\mathrm{F}_i$, so
$\Fm{-}\mathrm{F}_i$ is not essentially undecidable. The axiomatization of
$\Fm$ is thus minimal at each structural axiom, in the sense of
\cite[Ch.~II, Thm.~11]{TMR53} and \cite{HH14}.
\end{theorem}

\begin{proof}
\emph{Decidability.} $M_5$ is finite in a finite language, so
$\mathrm{Th}(M_5)$ is decidable by direct evaluation. In Presburger arithmetic
$(\mathbb{N},+)$ the following are parameter-free definable: $0$ (the unique
$x$ with $x+x=x$), the order ($x\le y\leftrightarrow\exists z\,(x+z=y)$), $1$
(the $\le$-least nonzero element), the successor ($y=x+1$), and
$\max$ (via $\le$); hence $M_4$, $M_7$ and $M_8$ are parameter-free definable
in $(\mathbb{N},+)$, whose theory is decidable \cite{Pre29}. The group
$(\mathbb{Z},+)$ with the constant $1$ is parameter-free interpretable in
$(\mathbb{N},+)$ by the standard difference construction (pairs modulo the
definable equivalence $(a,b)\sim(c,d):\Leftrightarrow a+d=b+c$, with
componentwise addition and $1$ represented by $(1,0)$), and the map
$x\mapsto 2x+1$ is definable from $+$ and $1$; hence $M_6$ is parameter-free
interpretable in $(\mathbb{N},+)$. In all cases Lemma~\ref{lem:transfer}
applies.

\emph{The structure $M_4$.} $\max$ is associative and commutative with
neutral element $0$ on $\mathbb{N}$, giving
$\mathrm{F}_1$--$\mathrm{F}_3$; $\mathrm{F}_5$--$\mathrm{F}_7$ are immediate
for the successor. $\mathrm{F}_8$: suppose
$\max(u,v)=\max(x{+}1,w)=:m$; then $m\ge x{+}1$, and $u=m$ or $v=m$; if
$u=m$, then $\max(x{+}1,u)=u$, so $w':=u$ witnesses the first disjunct (the
other case is symmetric). $\mathrm{F}_9$ with $\sq\,=\,\le$: for
($\Leftarrow$), if $t=\max(w,u{+}1)$ and $x\le u$ then $t\ge u{+}1>u\ge x$;
for ($\Rightarrow$), if $x\le t$ and $x\neq t$ then $x<t$, so $t\ge 1$, and
$u:=t{-}1$, $w:=0$ give $t=\max(0,u{+}1)$ with $x\le u$. $\mathrm{F}_4$
fails: $\max(1,2)=\max(2,2)$ but $1\neq 2$.

\emph{The structure $M_5$.} All operations are constant $0$, so every
equation between terms holds; the hypotheses and conclusions of
$\mathrm{F}_4$, $\mathrm{F}_6$--$\mathrm{F}_8$ are true outright (in
$\mathrm{F}_8$, $w':=0$), and both sides of $\mathrm{F}_9$ are true at the
only point $(0,0)$. $\mathrm{F}_5$ fails: $\enc{0}=0=\emptyset$.

\emph{The structure $M_6$.} $(\mathbb{Z},+,0)$ is an abelian group, giving
$\mathrm{F}_1$--$\mathrm{F}_4$; $2x{+}1$ is odd, hence nonzero
($\mathrm{F}_5$), and injective ($\mathrm{F}_7$). $\mathrm{F}_8$ holds
unconditionally: given the hypothesis, $w':=u-(2x{+}1)$ satisfies
$u=(2x{+}1)+w'$, so the first disjunct always holds. $\mathrm{F}_9$ with the
total relation: the left side is true for all $x,t$; the right side is true
as well, since $u:=0$, $w:=t{-}1$ give $t=w+(2\cdot 0{+}1)$ and $x\sq 0$
holds. $\mathrm{F}_6$ fails: $1+(-1)=0$ but $1\neq 0$. We note the
instructive converse: in a structure where every equation $u=a+w'$ is
solvable, the Levi axiom is vacuous; positivity $\mathrm{F}_6$ is precisely
what excludes such solvability in the intended model.

\emph{The structure $M_7$.} $\mathrm{F}_1$--$\mathrm{F}_4$ and $\mathrm{F}_6$
hold in $(\mathbb{N},+,0)$; $\enc{x}=1\neq 0$ gives $\mathrm{F}_5$.
$\mathrm{F}_8$: from $u+v=1+w$ we get $u+v\ge 1$, so $u\ge 1$ or $v\ge 1$;
if $u\ge 1$, then $w':=u-1$ gives $u=1+w'=\enc{x}\U w'$ (note
$\enc{x}=1$ for every $x$). $\mathrm{F}_9$: for $t=0$, both sides reduce to
$x=0$, the existential disjunct being refuted by $0\neq w+1$ in
$\mathbb{N}$; for $t\ge 1$, the left side is true, and so is the right, with
witnesses $w:=t{-}1$ and $u:=1$, since $t=(t{-}1)+\enc{1}$ and $x\sq 1$
holds for all $x$. $\mathrm{F}_7$ fails: $\enc{0}=\enc{1}=1$ but
$0\neq 1$.

\emph{The structure $M_8$.} $\mathrm{F}_1$--$\mathrm{F}_7$ hold as in
$(\mathbb{N},+)$ with the successor. $\mathrm{F}_9$ with $\sq\,=\,\le$: for
($\Leftarrow$), $t=w+(u{+}1)\ge u{+}1>u\ge x$; for ($\Rightarrow$), if
$x<t$ then $u:=t{-}1$, $w:=0$ give $t=0+(u{+}1)$ and $x\le u$.
$\mathrm{F}_8$ fails: $2+2=\enc{2}+1$, i.e.\ the hypothesis holds with
$u=v=2$, $x=2$, $w=1$, but $2=3+w'$ has no solution in $\mathbb{N}$.
\end{proof}

\subsection{Conservativity of the containment axiom}

\begin{proposition}\label{prop:cons}
Every $\{\emptyset,\enc{\cdot},\U\}$-structure satisfying
$\mathrm{F}_1$--$\mathrm{F}_8$ admits an expansion to a model of $\Fm$.
Consequently, $\Fm$ is conservative over $\mathrm{F}_1$--$\mathrm{F}_8$ for
$\sq$-free sentences: if $\Fm\vdash\varphi$ and $\varphi$ does not contain
$\sq$, then $\mathrm{F}_1\text{--}\mathrm{F}_8\vdash\varphi$.
\end{proposition}

\begin{proof}
Let $M_0\models\mathrm{F}_1$--$\mathrm{F}_8$ and define, on the complete
lattice of binary relations on $M_0$ ordered by inclusion, the operator
\[
\Gamma(R)\ :=\ \{(x,t)\ :\ x=t\ \vee\ \exists w,u\,
\bigl(t=w\U\enc{u}\wedge(x,u)\in R\bigr)\}.
\]
$R$ occurs only positively, so $\Gamma$ is monotone, and by the
Knaster--Tarski theorem \cite{Tar55} it has a least fixed point
$\sq^{*}$. The axiom $\mathrm{F}_9$ asserts precisely that $\sq$ is a fixed
point of $\Gamma$, so $(M_0,\sq^{*})\models\Fm$, the axioms
$\mathrm{F}_1$--$\mathrm{F}_8$ being unaffected by the expansion.
Conservativity follows by contraposition: a model of
$\mathrm{F}_1$--$\mathrm{F}_8+\neg\varphi$ expands to a model of
$\Fm+\neg\varphi$.
\end{proof}

\begin{remark}\label{rem:intended}
In the intended reduct $(\HFM;\emptyset,\enc{\cdot},\U)$ the least fixed
point $\sq^{*}$ coincides with the containment relation of
Section~\ref{sec:theories}. Indeed, the rank-recursive definition of $\sq$
makes it a fixed point of $\Gamma$; and $\sq$ is contained in every fixed
point $R$, by induction on the rank of $t$: if $x\sq t$ with $x\neq t$,
there is $u$ with $u\in_1 t$, of smaller rank, and $x\sq u$; the induction
hypothesis gives $(x,u)\in R$, whence $(x,t)\in\Gamma(R)=R$. Thus
$\mathrm{F}_9$ names the canonical relation and, by
Proposition~\ref{prop:cons}, constrains the algebraic reduct not at all.
\end{remark}

\subsection{The delimiting axioms}

\begin{remark}\label{rem:frame}
The monoid axioms $\mathrm{F}_1$--$\mathrm{F}_3$ play a different role from
the structural axioms: they fix the algebraic regime within which the
question is posed, namely that juxtaposition forms multisets --- unordered,
ungrouped, with an empty aggregate. Weakening this regime does not weaken the
theory but changes the subject: without commutativity and associativity the
finitary aggregates are strings, sequences, and ordered trees, and the
interpretability profile of their weak theories is exactly the charted
territory of
\cite{Grz05,GZ08,HH14,KM20,Dam22,Dam23,Mur22,Mur24,KM24}. Within the regime
fixed by $\mathrm{F}_1$--$\mathrm{F}_3$, Theorems~\ref{thm:A}
and~\ref{thm:B} together with Theorem~\ref{thm:indep} give the complete
strength profile of the remaining axioms.
\end{remark}

\section{The calculus of indications}\label{sec:calculus}

The intended model $\mm$ has an independent pedigree. The \emph{forms} of
Spencer-Brown's calculus of indications \cite{SB69} are the expressions
generated from the empty expression by enclosure and juxtaposition; their
formal study includes the completeness and decidability of the primary
algebra \cite{Ban77,Meg03}, Varela's extended calculus \cite{Var75}, and the
waveform reading of re-entry \cite{KV80}. This section identifies the forms,
modulo the congruence appropriate to spatial juxtaposition, with hereditarily
finite multisets (Theorems~\ref{thm:ident1} and~\ref{thm:ident2}), records
the interpretability profile of the condensed variant
(Proposition~\ref{prop:condensed}), shows that bounded re-entry remains
within a decidable regime (Proposition~\ref{prop:reentry}), and assembles the
resulting boundary map. The section depends on
Sections~\ref{sec:prelim}--\ref{sec:theories} for definitions, quotes
Theorems~\ref{thm:A} and~\ref{thm:B} only in its final paragraph, and nothing
elsewhere in the paper depends on it.

\subsection{Forms and the identification theorems}
Let $\mathrm{Fm}_0$ be the smallest set of expressions containing the empty
expression $\epsilon$ and closed under enclosure $f\mapsto\enc{f}$ and
juxtaposition $(f,g)\mapsto fg$. Expressions are written in a common space:
juxtaposition carries neither order nor grouping, and the empty expression is
an expression. Accordingly, let $\approx$ be the congruence on
$\mathrm{Fm}_0$ (with respect to both operations) generated by
\[
(fg)h\approx f(gh),\qquad fg\approx gf,\qquad f\epsilon\approx f .
\]
The \emph{subform} relation is defined recursively, mirroring $\mathrm{F}_9$:
$f$ is a subform of $g$ iff $f\approx g$, or $g\approx h\enc{u}$ for some
$h,u$ such that $f$ is a subform of $u$. The \emph{depth} of a form is the
maximal nesting of enclosures; it is invariant under $\approx$, since the
generating relations preserve the enclosure structure at every level.

\begin{lemma}\label{lem:constituents}
Every form is $\approx$-equivalent to a juxtaposition
$\enc{f_1}\cdots\enc{f_k}$ of enclosed forms, $k\ge 0$ (the case $k=0$ being
$\epsilon$), with each $f_i$ of strictly smaller depth.
\end{lemma}

\begin{proof}
Structural induction: $\epsilon$ is the empty juxtaposition; $\enc{f}$ is
such a juxtaposition with $k=1$; and if
$f\approx\enc{f_1}\cdots\enc{f_k}$ and $g\approx\enc{g_1}\cdots\enc{g_l}$,
then $fg\approx\enc{f_1}\cdots\enc{f_k}\enc{g_1}\cdots\enc{g_l}$ by
congruence and associativity. The depth claim is immediate, the depth of
$\enc{f_i}$ being $\mathrm{depth}(f_i)+1$.
\end{proof}

\begin{theorem}[Identification, de-condensed]\label{thm:ident1}
The recursion $\iota(\epsilon):=\emptyset$,
$\iota(\enc{f}):=\enc{\iota(f)}$, $\iota(fg):=\iota(f)\U\iota(g)$ induces an
isomorphism of $\{\emptyset,\enc{\cdot},\U\}$-algebras
\[
\iota\colon\ \mathrm{Fm}_0/{\approx}\ \xrightarrow{\ \cong\ }\
(\HFM;\,\emptyset,\enc{\cdot},\U),
\]
which moreover carries the subform relation to the containment relation
$\sq$.
\end{theorem}

\begin{proof}
The three generating relations of $\approx$ hold in $\mm$ under $\iota$,
since $\U$ is associative and commutative with neutral element $\emptyset$;
hence $\iota$ is well defined on $\approx$-classes and is a homomorphism.
\emph{Surjectivity}, by induction on rank: $\emptyset=\iota(\epsilon)$, and
if $m=[m_1,\dots,m_k]$ with $m_i=\iota(f_i)$, then
$m=\iota(\enc{f_1}\cdots\enc{f_k})$. \emph{Injectivity}, by induction on
depth: let $\iota(f)=\iota(g)$; by Lemma~\ref{lem:constituents},
$f\approx\enc{f_1}\cdots\enc{f_k}$ and $g\approx\enc{g_1}\cdots\enc{g_l}$,
so $[\iota(f_1),\dots,\iota(f_k)]=[\iota(g_1),\dots,\iota(g_l)]$ as
multisets. Hence $k=l$, and there is a bijection $\sigma$ of indices with
$\iota(f_i)=\iota(g_{\sigma(i)})$. The $f_i$ and $g_j$ have strictly smaller
depth, so the induction hypothesis gives $f_i\approx g_{\sigma(i)}$ for all
$i$, and commutativity with congruence reassembles
$f\approx\enc{f_1}\cdots\enc{f_k}
\approx\enc{g_{\sigma(1)}}\cdots\enc{g_{\sigma(k)}}
\approx\enc{g_1}\cdots\enc{g_l}\approx g$.
\emph{The relation}: the subform relation and $\sq$ are the least fixed
points of the monotone operators
\begin{gather*}
\Gamma_{\mathrm{Fm}}(R):=\{(f,g):f\approx g\ \vee\ \exists h,u\,
(g\approx h\enc{u}\wedge(f,u)\in R)\},\\
\Gamma_{\HFM}(S):=\{(x,t):x=t\ \vee\ \exists w,u\,
(t=w\U\enc{u}\wedge(x,u)\in S)\},
\end{gather*}
the former by definition, the latter by Remark~\ref{rem:intended}. Since
$\iota$ is an isomorphism of the algebraic reducts, the lattice isomorphism
$R\mapsto\iota[R]$ between the complete lattices of binary relations
intertwines $\Gamma_{\mathrm{Fm}}$ with $\Gamma_{\HFM}$, and therefore
matches their least fixed points.
\end{proof}

\begin{theorem}[Identification, condensed]\label{thm:ident2}
Let $\approx_c$ be the congruence generated by $\approx$ together with the
idempotence law $ff\approx_c f$ (iteration, consequence C5 of
\cite{SB69}).
Then $\mathrm{Fm}_0/{\approx_c}$ is isomorphic to
$(\mathrm{HF};\,\emptyset,\{\cdot\},\cup)$, the algebra of hereditarily
finite sets with singleton and union; under this isomorphism the composite
operation $x\enc{y}$ is set adjunction $x\cup\{y\}$.
\end{theorem}

\begin{proof}
By Theorem~\ref{thm:ident1} it suffices to show that the quotient of $\HFM$
by the congruence $\theta$ generated by $m\U m=m$ is isomorphic to
$\mathrm{HF}$. Define the \emph{hereditary support}
$\mathrm{supp}\colon\HFM\to\mathrm{HF}$ by rank recursion:
$\mathrm{supp}(m):=\{\mathrm{supp}(u):u\in_1 m\}$. Then $\mathrm{supp}$ is a
homomorphism --- $\mathrm{supp}(\emptyset)=\emptyset$,
$\mathrm{supp}(\enc{x})=\{\mathrm{supp}(x)\}$,
$\mathrm{supp}(x\U y)=\mathrm{supp}(x)\cup\mathrm{supp}(y)$ --- and it is
surjective by $\in$-induction on $\mathrm{HF}$. Since
$\mathrm{supp}(m\U m)=\mathrm{supp}(m)$, we get
$\theta\subseteq\ker(\mathrm{supp})$. For the converse, define for
$s\in\mathrm{HF}$ the canonical representative
$\hat c(s):=[\hat c(s_1),\dots,\hat c(s_j)]$, where $s_1,\dots,s_j$ are the
distinct elements of $s$. We claim every $m$ is $\theta$-congruent to
$\hat c(\mathrm{supp}(m))$; granting the claim, $\mathrm{supp}(x)=
\mathrm{supp}(y)$ implies
$x\mathrel{\theta}\hat c(\mathrm{supp}(x))=\hat c(\mathrm{supp}(y))
\mathrel{\theta}y$, so $\ker(\mathrm{supp})\subseteq\theta$. The claim
follows by rank induction: writing $m=[u_1,\dots,u_k]$, congruence with
respect to $\enc{\cdot}$ and $\U$ together with the induction hypothesis
gives $m\mathrel{\theta}
[\hat c(\mathrm{supp}(u_1)),\dots,\hat c(\mathrm{supp}(u_k))]$, and
collapsing duplicated entries by instances
$\enc{a}\U\enc{a}\mathrel{\theta}\enc{a}$ of idempotence leaves exactly one
occurrence of $\hat c(t)$ for each distinct $t\in\mathrm{supp}(m)$, i.e.\
$\hat c(\mathrm{supp}(m))$. Finally,
$\mathrm{supp}(x\U\enc{y})=\mathrm{supp}(x)\cup\{\mathrm{supp}(y)\}$ is the
adjunction identity.
\end{proof}

\subsection{The condensed theory}
Let $\Fm_{\!c}$ be the finitely axiomatized theory in the language
$\{\emptyset,\enc{\cdot},\U\}$ whose axioms are $\mathrm{F}_1$,
$\mathrm{F}_2$, $\mathrm{F}_3$, $\mathrm{F}_5$, $\mathrm{F}_6$,
$\mathrm{F}_7$ of Figure~\ref{fig:F} together with
\[
(\mathrm{Id})\ \ x\U x=x,
\qquad
(\mathrm{M})\ \ \forall x,y,u\;\bigl(u\in_1(x\U\enc{y})\ \leftrightarrow\
u\in_1 x\ \vee\ u=y\bigr),
\]
where $u\in_1 z$ abbreviates $\exists w\,(z=w\U\enc{u})$. In
$(\mathrm{HF};\emptyset,\{\cdot\},\cup)$ the formula $u\in_1 z$ defines
membership --- if $u\in z$ then $z=z\cup\{u\}$, and conversely
$z=w\cup\{u\}$ puts $u$ into $z$ --- and all axioms of $\Fm_{\!c}$ hold
there, $(\mathrm{M})$ being the defining property of adjunction.

\begin{proposition}\label{prop:condensed}
$\Fm_{\!c}\mutint\Qr$.
\end{proposition}

\begin{proof}
Given the chain $\Qr\mutint\AST$ of Theorem~\ref{thm:chains}, the first half
is a transcription. \emph{$\Qr\intp\Fm_{\!c}$}: translate the language of
$\AST$ by sending $\in$ to $\in_1$, with trivial domain and equality as
equality. The empty-set axiom $\exists y\,\forall x\,\neg(x\in y)$ holds with
$y:=\emptyset$: from $\emptyset=w\U\enc{x}$, axioms $\mathrm{F}_2$,
$\mathrm{F}_6$, $\mathrm{F}_5$ yield a contradiction, exactly as in
Fact~\ref{fact:one}, whose proof uses only these axioms. The adjunction axiom
$\forall x,y\,\exists z\,\forall u\,(u\in z\leftrightarrow u\in x\vee u=y)$
holds with $z:=x\U\enc{y}$ by $(\mathrm{M})$. Hence $\AST\intp\Fm_{\!c}$, and
composition with $\Qr\intp\AST$ gives $\Qr\intp\Fm_{\!c}$.

\emph{$\Fm_{\!c}\intp\Qr$}: adapt the interpretation of
Section~\ref{sec:thmB}. Call a word \emph{set-canonical} if it is balanced,
the sequence of top-level blocks of the word and of every block interior is
sorted with pairwise distinct entries, and all interiors are set-canonical;
replace insertion by \emph{absorbing} insertion
($\mathrm{Ins}_c(w,b):=w$ if $\mathrm{Occ}(w,b)$, else $\mathrm{Ins}(w,b)$)
and $\mathrm{Mrg}$ by the corresponding iteration $\mathrm{Mrg}_c$. The entry
pool of $\mathrm{Mrg}_c(x,y)$ is the union of the two entry pools, with a
$\Sigma^b_1$-witnessing correspondence, and Lemma~\ref{lem:rigid} applies
verbatim to sorted sequences with pairwise distinct entries; the proofs of
Lemmas~\ref{lem:blocks} and~\ref{lem:book} and of
Proposition~\ref{prop:F18} for $\mathrm{F}_1$, $\mathrm{F}_2$,
$\mathrm{F}_3$, $\mathrm{F}_5$, $\mathrm{F}_7$ then go through unchanged.
Three points differ. $(\mathrm{Id})$: in $\mathrm{Mrg}_c(x,x)$ every
insertion is absorbed, so $\mathrm{Mrg}_c(x,x)=x$. $\mathrm{F}_6$: length
additivity becomes subadditivity, but the entries of $x$ remain entries of
$\mathrm{Mrg}_c(x,y)$, so $\mathrm{Mrg}_c(x,y)=\varepsilon$ forces
$x=\varepsilon$. $(\mathrm{M})$: the translation of $u\in_1 v$ is equivalent
to $\mathrm{Occ}(v,\mathrm{Wr}(u))$ --- for the nontrivial direction take
$w:=\mathrm{Rem}(v,\mathrm{Wr}(u))$, which is set-canonical, and reinsertion
restores $v$ by Lemma~\ref{lem:book}(iii), absorption being vacuous after
removal of the unique occurrence --- and the entry pool of
$\mathrm{Mrg}_c(x,\mathrm{Wr}(y))$ is that of $x$ together with
$\mathrm{Wr}(y)$, where $\mathrm{Wr}(u)=\mathrm{Wr}(y)$ iff $u=y$. Composing
with $S^1_2\mutint\Qr$ (Theorem~\ref{thm:chains}) completes the proof.
\end{proof}

\begin{remark}
Proposition~\ref{prop:condensed} is recorded for completeness of the map:
modulo Theorem~\ref{thm:ident2}, its first half is a notational
transcription of the published equivalence $\AST\mutint\Qr$
\cite{TMR53,CH70,MM94,Nel86,Dam17}. The multiset-theoretic content of the
paper resides in Sections~\ref{sec:pairing}--\ref{sec:minimality}.
\end{remark}

\subsection{Bounded re-entry}
Chapter~11 of \cite{SB69} introduces equations of the second degree, in
which an expression re-enters its own indicational space; the tradition of
\cite{Var75,KV80} reads their solutions as oscillations in discrete time. We
record the finite-state content of this mechanism. A \emph{re-entry system
of dimension $n$} is an $n$-tuple $E_1,\dots,E_n$ of primary-algebra
expressions in the variables $x_1,\dots,x_n$, defining the synchronous
update map $F\colon\{0,1\}^n\to\{0,1\}^n$ under the two-element evaluation
of the primary algebra (juxtaposition as join, enclosure as complement)
\cite{Ban77,Meg03}; a \emph{trajectory} is a sequence
$(F^{k}(a))_{k\in\mathbb{N}}$ for an initial state $a$. Since
$\{\vee,\neg\}$ is functionally complete --- conjunction being definable by
De Morgan --- the re-entry systems of dimension $n$ are precisely the
$n$-dimensional synchronous Boolean networks.

\begin{proposition}\label{prop:reentry}
\textup{(i)} Every trajectory is ultimately periodic, with preperiod and
period jointly bounded by $2^{n}$.
\textup{(ii)} For every formula $\varphi$ of monadic second-order logic over
$(\mathbb{N},\mathrm{succ})$ with free set variables
$X_1,\dots,X_n$, it is decidable whether the trajectory of a given re-entry
system from a given initial state satisfies $\varphi$ when $X_i$ is read as
the $i$-th coordinate of the trajectory; in particular, equality of the
trajectories of two given systems, from given or from all initial states, is
decidable.
\end{proposition}

\begin{proof}
(i) The state space has $2^{n}$ elements and $F$ is a function, so two of
the states $a$, $F(a)$, \dots, $F^{2^{n}}(a)$ coincide, and determinism closes the
trajectory into a cycle thereafter. (ii) By (i), each coordinate of the
trajectory is an ultimately periodic subset of $\mathbb{N}$, with preperiod
and period computable from the system by direct simulation; such a set is
definable over $(\mathbb{N},\mathrm{succ})$ by an explicit monadic
second-order formula, listing the finitely many exceptional positions
outright and describing the periodic tail by residues. Substituting these
definitions for the free set variables turns $\varphi$ into a sentence of
S1S, which is decidable by B\"uchi's theorem \cite{Buc62}; cf.\ also
\cite{Elg61}.
\end{proof}

\subsection{The boundary}
The results now in place locate the boundary of essential undecidability
within the calculus of indications. Below it lie: the primary algebra, whose
two-element collapse is complete and decidable \cite{Ban77,Meg03}, together
with its many-valued completions \cite{Var75}; multiplicity without nesting
--- flat multisets over a support of size $k$ form $(\mathbb{N}^{k},+)$,
decidable by \cite{Pre29} with \cite{FV59}, while over a countably infinite
support they form, via prime factorization, the multiplicative monoid of the
positive integers, whose decidability is Skolem's arithmetic
\cite{Sko30,Mos52}; ordered nesting without a containment relation --- the
first-order theories of free term algebras are decidable
\cite{Mal61,Mah88,Hod93}; and finite-state re-entry
(Proposition~\ref{prop:reentry}). Above it lie the hereditary containment
structures: hereditarily finite sets with adjunction, i.e.\ the condensed
forms (Theorem~\ref{thm:ident2}, Proposition~\ref{prop:condensed}), and
hereditarily finite multisets with containment, i.e.\ the forms themselves
(Theorem~\ref{thm:ident1}), whose schematic theory lies in the
mutual-interpretability class of $\Rr$ and whose finitely axiomatized theory
lies in that of $\Qr$ (Theorems~\ref{thm:A} and~\ref{thm:B}). The arithmetic
latent in the calculus of indications is thus, up to mutual
interpretability, exactly Robinson's $\Qr$; and it is activated neither by
iteration, nor by multiplicity, nor by feedback, but by unbounded
discriminability of form within form.

\section*{Acknowledgements}
During the preparation of this manuscript the author used Claude (Anthropic) as a support tool for literature search, language editing, LaTeX drafting, bibliographic checking, and assistance with the formal expression of mathematical arguments. The research questions, conceptual development, mathematical constructions, theorem statements, proofs, and all scientific conclusions originated from the author. The author independently evaluated all AI-generated suggestions, accepted or rejected them as appropriate, and assumes full responsibility for every aspect of the manuscript.

\end{document}